\documentclass[11pt,twoside,openright]{book}
\AtBeginDocument{\addtocontents{toc}{\protect\thispagestyle{plain}}}
\usepackage[english]{babel}
\usepackage[ansinew]{inputenc}
\usepackage{amsmath}
\usepackage{amssymb}
\usepackage{amsthm}
\usepackage{graphicx}
\usepackage{url}
\usepackage{float}
\usepackage{multirow}
\usepackage{dsfont}
\usepackage{appendix}
\usepackage{booktabs}
\usepackage{setspace}
\usepackage{fancyhdr}
\usepackage{tikz}
\usetikzlibrary{calc}
\usepackage{hyperref} %UNCOMMENT HYPERREF PACKAGE WHEN CREATING FINAL PDF
\tikzstyle{vertex} = [circle, draw, inner sep=0pt, minimum size=6pt]
\tikzstyle{edge} = [draw, -]
\pgfdeclarelayer{background}
\pgfdeclarelayer{foreground}
\pgfsetlayers{background,main,foreground}
\fancypagestyle{plain}{ %
  \fancyhf{} % remove everything
  \fancyhead[RO]{\thepage}
}

\let\origdoublepage\cleardoublepage
\newcommand{\clearemptydoublepage}{%
  \clearpage
  {\pagestyle{empty}\origdoublepage}%
}

\let\cleardoublepage\clearemptydoublepage

\newtheorem{lemma}{Lemma}
\newtheorem{theorem}[lemma]{Theorem}
\newtheorem{defn}[lemma]{Definition}

\newtheorem{claim}{Claim}
\newtheorem{obs}[lemma]{Observation}

\newtheorem{prob}{Problem}
\newtheorem{conjecture}[prob]{Conjecture}

\begin{document}

% UNCOMMENT HYPERREF PACKAGE WHEN CREATING FINAL PDF

\frontmatter

    \title{On local search and LP and SDP relaxations for $k$-set packing}
    \author{Tim Oosterwijk \\ Eindhoven University of Technology}
    \date{October 8, 2013}
    %\maketitle

%\setcounter{page}{1}
%\pagenumbering{roman}

\begin{titlepage}
    \let\footnotesize\small
    \let\footnoterule\relax
    \let \footnote \thanks
    \setcounter{footnote}{0}
    \null\vfil
   % \vskip 60\p@
    \begin{center}
      \setlength{\parskip}{0pt}
      {\large\textbf{Eindhoven University of Technology}\par}
      \vfill
      {\huge \bf On local search and LP and SDP relaxations for $k$-set packing \par}
      \vfill
      {\LARGE Tim Oosterwijk \par}
      \vfill
      {\large A thesis submitted in partial fulfillment\par}
      {\large of the requirements of the degree of\par}
      %\bigskip
      \bigskip
      %{\large in \par}
      {\large \emph{\textbf{Master of Science}} \par}
      \bigskip
      \bigskip
      \bigskip
      {\Large October 8, 2013 \par}
      \bigskip
    \end{center}
    \par
    %\@thanks
    \vfil\null
  \end{titlepage} 
\cleardoublepage

\section*{Important note}

There is a mistake in the following line of Theorem 17: ``As an induced subgraph of $H$ with more edges than vertices constitutes an improving set''. Therefore, the proofs of Theorem 17, and hence Theorems 19, 23 and 24, are false. It is still open whether these theorems are true.

\cleardoublepage

\chapter*{\centering \begin{normalsize}Abstract\end{normalsize}}

\thispagestyle{plain}
\addcontentsline{toc}{chapter}{\numberline{}Abstract}

\begin{quotation}
\noindent Set packing is a fundamental problem that generalises some well-known combinatorial optimization problems and knows a lot of applications. It is equivalent to hypergraph matching and it is strongly related to the maximum independent set problem.

In this thesis we study the $k$-set packing problem where given a universe $\mathcal{U}$ and a collection $\mathcal{C}$ of subsets over $\mathcal{U}$, each of cardinality $k$, one needs to find the maximum collection of mutually disjoint subsets. Local search techniques have proved to be successful in the search for approximation algorithms, both for the unweighted and the weighted version of the problem where every subset in $\mathcal{C}$ is associated with a weight and the objective is to maximise the sum of the weights. We make a survey of these approaches and give some background and intuition behind them. In particular, we simplify the algebraic proof of the main lemma for the currently best weighted approximation algorithm of Berman (\cite{Berman}) into a proof that reveals more intuition on what is really happening behind the math.

The main result is a new bound of $\frac{k}{3} + 1 + \varepsilon$ on the integrality gap for a polynomially sized LP relaxation for $k$-set packing by Chan and Lau (\cite{LapChiLau}) and the natural SDP relaxation \textbf{[NOTE: see page iii]}. We provide detailed proofs of lemmas needed to prove this new bound and treat some background on related topics like semidefinite programming and the Lov\'{a}sz Theta function.

Finally we have an extended discussion in which we suggest some possibilities for future research. We discuss how the current results from the weighted approximation algorithms and the LP and SDP relaxations might be improved, the strong relation between set packing and the independent set problem and the difference between the weighted and the unweighted version of the problem.

%Finally we discuss some directions in which the current results might be improved, the issues one encounters when mimicking the unweighted approximation algorithms in the weighted problem and why the $\frac{k+1}{3}$-approximations require much more work than the $\frac{k+2}{3}$-approximations.

%The main result is a new bound of $\frac{k}{3} + 1 + \varepsilon$ on the integrality gap for an LP relaxation for $k$-set packing by Chan and Lau (\cite{LapChiLau}), known as the intersecting family LP. We provide detailed proofs of lemmas needed to prove this new bound and show that the LP can be written in a size polynomial in the input. Then we treat some background on semidefinite programming and show that there exists a polynomially sized SDP for $k$-set packing with the same improved bound on the integrality gap.
\end{quotation} 
\cleardoublepage
\chapter*{\centering \begin{normalsize}Acknowledgements\end{normalsize}}

\thispagestyle{plain}
\addcontentsline{toc}{chapter}{\numberline{}Acknowledgements}

\begin{quotation}
\noindent First of all, I would like to give a huge thanks to my supervisor Nikhil Bansal. I learned so much during (the research for) this thesis about subjects that interest me, more than this thesis could ever contain. You have only encouraged me to absorb all that knowledge and broaden my horizons. You showed me fascinating results on all related topics and problems and inspired me by researching all that might be interesting for my academic future. Thank you for everything you have done to make me a better researcher.

I would like to thank my future (co-)promotor Tjark Vredeveld; your interest in the thesis and me has contributed to my confidence as a researcher and my general feeling of well-being. Thank you, I'm looking forward to our cooperation in the years to come.

I also owe a thank you to my fellow student Annette Ficker, with whom I initially researched various interesting problems until we parted ways when I chose $k$-set packing as my subject and you chose 2-dimensional bin packing for your thesis. Thank you for the fruitful discussions about everything we encountered and of course for the very nice time we had together during university.

I also would like to thank some other people who discussed the problem with me, among which are authors of some papers I read. In particular I would like to thank Per Austrin, Fabrizio Grandoni, Konstantin Makarychev, Monaldo Mastrolilli, Viswanath Nagarajan and Ruben van der Zwaan for their time.

A big thank you to the teachers in the combinatorial optimization group for their interesting and inspiring classes: Nikhil Bansal, Cor Hurkens, Judith Keijsper, Rudi Pendavingh and Gerhard Woeginger. A special thanks to Jan Draisma and Gerhard Woeginger for forming the assessment committee for my thesis.

Also I want to thank some other fellow students, in particular Jorn van der Pol and Reint den Toonder, for their interest in the thesis, their general support and of course their friendly company for the past years.

Finally a thanks to my family and boyfriend for their support and faith. I would not be where I am today if it weren't for you. A huge thanks to all of you.
\end{quotation} 
\cleardoublepage

% UNCOMMENT HYPERREF PACKAGE WHEN CREATING FINAL PDF

\newpage
\phantomsection
\addcontentsline{toc}{chapter}{\numberline{}Table of contents}
\setcounter{tocdepth}{1}
\tableofcontents

\mainmatter

%\pagenumbering{arabic} This is after \chapter{Introduction}
%\setcounter{page}{1} This is after \chapter{Introduction}

\chapter{Introduction}\label{chap:Introduction}

%\pagenumbering{arabic}
%\setcounter{page}{1}

\section{Definition of the problem}\label{sec:Problem}

\paragraph{Set packing and $k$-set packing} Set packing is one of Karp's 21 NP-complete problems \cite{Karp} and it has received a lot of attention during the past years. A lot of progress has been made on the complexity of this problem, even though under standard complexity assumptions algorithms for this problem require at least superpolynomial or perhaps even exponential running time. Set packing is a fundamental problem that generalises some well-known problems and thus knows a lot of applications. There has been a long line of research on this problem. In this thesis we will consider $k$-set packing, which is the following problem.
\begin{quote}
\begin{bf}$k$-Set Packing ($k$-SP)\end{bf}\\
Given: a universe $\mathcal{U}$ of $N$ elements and a collection $\mathcal{C} \subseteq 2^\mathcal{U}$ of $n$ subsets over $\mathcal{U}$, each of cardinality $k$ \\
Find: a maximum collection of mutually disjoint subsets in $\mathcal{C}$.
\end{quote}
Any collection of mutually disjoint subsets in $\mathcal{C}$ is called a packing and the goal is to find the largest packing. In $k$-set packing every subset in $\mathcal{C}$ contains at most $k$ elements. Without loss of generality assume every subset contains exactly $k$ elements: add some unique dummy elements to the subsets of less than $k$ elements.

$k$-set packing is a special case of the optimization version of the set packing problem, where there is no restriction on the cardinality of every subset in $\mathcal{C}$.
\begin{quote}
\begin{bf}(Maximum) Set Packing (SP)\end{bf}\\
Given: a universe $\mathcal{U}$ of $N$ elements and a collection $\mathcal{C} \subseteq 2^\mathcal{U}$ of $n$ subsets over $\mathcal{U}$ \\
Find: a maximum collection of mutually disjoint subsets in $\mathcal{C}$.
\end{quote}

\section{Terminology}\label{sec:Terminology}

%Let $G = (V,E)$ be a graph and define for a vertex $v \in V$ its open neighbourhood (or just its neighbourhood) $N_G(v) = \{ w \in V \mid \{v,w\} \in E \}$ and its closed neighbourhood $N_G[v] = N_G(v) \cup \{v\}$. For a subset of the vertices $W \subseteq V$ we define its closed neighbourhood $N_G[W] = \bigcup_{w \in W} N_G[w]$ and we write $N_G(W) = N_G[W] \setminus W$ for the (open) neighbourhood of $W$. For two given subsets of the vertices $X$ and $Y$ we write $N_G(X,Y) = N(X) \cap Y$, i.e. the neighbourhood of $X$ in $Y$. We write $N_G(x,Y)$ for $N_G(\{x\},Y)$ for some vertex $x$ and some subset of the vertices $Y$. When the context allows it, we will drop the subscript $G$ and just use $N(v)$, $N[v]$, $N[W]$, $N(W)$, $N(X,Y)$ and $N(x,Y)$.

We proceed with some notational conventions and definitions to make sure no confusion may arise.
By $\mathcal{U}$ we mean the universe of elements over which the $k$-set system $\mathcal{C} \subseteq 2^{\mathcal{U}}$ has been defined. We write $|\mathcal{U}| = N$ and $|\mathcal{C}| = n$. By $I = (\mathcal{U},\mathcal{C})$ a given instance for $k$-set packing is denoted in which the objective is to find the largest packing. %Then define the following.
\begin{defn}
\textbf{(Conflict graph)} Let an instance $I = (\mathcal{U},\mathcal{C})$ for the $k$-set packing problem be given. The conflict graph of $I$ is the graph $G$ where every subset in $\mathcal{C}$ is represented by a vertex. Two vertices %$S,T \in V(G)$
are adjacent if and only if %$\mathcal{S} \cap \mathcal{T} \neq \emptyset$, i.e. when the two sets $\mathcal{S}, \mathcal{T} \in \mathcal{C}$ these vertices correspond to are in conflict.
the subsets in $\mathcal{C}$ these vertices correspond to intersect each other.
\end{defn}
Now let $\mathcal{A}$ be some solution to this instance, i.e. a collection of subsets in $\mathcal{C}$ that are mutually disjoint.
\begin{defn}
\textbf{(Intersection graph)} Let an instance $I$ for the $k$-set packing problem be given and let $\mathcal{A}$ and $\mathcal{B}$ be two packings. The intersection graph of $\mathcal{A}$ and $\mathcal{B}$ is the induced subgraph of their vertices in the conflict graph of $I$.
\end{defn}
The bipartite intersection graph of $\mathcal{A}$ and $\mathcal{B}$ thus contains a vertex for every set in $\mathcal{A}$ and $\mathcal{B}$ and two vertices %$a \in A$ and $b \in B$
are adjacent % in $G$
if and only if %$a \cap b \neq \emptyset$
the subsets in the set packing instance are in conflict. Throughout this thesis calligraphic letters $\mathcal{A}$ will be used to denote collections of subsets and normal letters $A$ to denote the set of vertices corresponding to $\mathcal{A}$ in the conflict or intersection graph. Denote $N_G(B,A) = N(B) \cap A$, the neighbours of $B$ in $A$ in their intersection graph $G$. Sometimes for brevity we just write $N(G,A)$ if the graph $G$ is clear from the context. We can now succinctly define an improving set.
\begin{defn}
\textbf{(Improving set)} Let $\mathcal{A}$ and $\mathcal{B}$ be two packings. $\mathcal{B}$ is called an improving set for $\mathcal{A}$ when $\left| A \cup B \setminus N_G(B,A) \right| > |A|$, i.e. when adding the sets of $\mathcal{B}$ to the current solution $\mathcal{A}$ and removing the sets of $\mathcal{A}$ that they intersect, leads to a solution of larger cardinality.
\end{defn}
%
%Let $\mathcal{B}$ be a collection of mutually disjoint subsets in $\mathcal{C}$ for now as well and consider the conflict graph $G$ of the instance $I$, in which $\mathcal{A}$ and $\mathcal{B}$ are subsets of the vertices called $A$ and $B$. We say that $\mathcal{B}$ improves $\mathcal{A}$ or that $\mathcal{B}$ is an improving set for $\mathcal{A}$ when $\left| A \cup B \setminus N_G(B,A) \right| > |A|$, i.e. when adding the sets of $\mathcal{B}$ to the current solution $\mathcal{A}$ and removing the sets of $\mathcal{A}$ that they intersect, leads to a solution of larger cardinality.
The removal of the neighbourhood of $B$ in $A$ ensures that the new solution is also mutually disjoint. %By adding an improving set $\mathcal{B}$ to $\mathcal{A}$ we mean replacing $A$ by the sets corresponding to $A \cup B \setminus N_G(B,A)$ in $G$.
Now define the following.

\begin{defn}
\textbf{($t$-locally optimal solution)} Let $\mathcal{A}$ be a packing. $\mathcal{A}$ is said to be $t$-locally optimal when for every collection of mutually disjoint subsets $\mathcal{B}$ with $|\mathcal{B}| \leq t$ we have $\left| A \cup B \setminus N_G(B,A) \right| \leq |A|$. In other words, a solution is a $t$-locally optimal solution when there does not exist an improving set of size at most $t$.
\end{defn}

Finally, an algorithm is said to approximate a maximization problem within a factor $\rho>1$ if for every instance of the problem $\frac{v(OPT)}{v(A)} \leq \rho$, where $v(A)$ is the objective value of the output of the algorithm and $v(OPT)$ is the best achievable objective value for that instance. We call $\rho$ its approximation guarantee.

%We say a solution $\mathcal{A}$ is $t$-local optimal when for every collection of mutually disjoint subsets $\mathcal{B} \subseteq \mathcal{C}$ with $|\mathcal{B}| \leq t$ we have $\left| A \cup B \setminus N_G(B,A) \right| \leq |A|$. In other words, a solution is a $t$-local optimal solution when there does not exist an improving set of size at most $t$.

\section{Contribution}\label{sec:Contribution}

%Set packing and $k$-set packing are NP-complete, but $k$-set packing is fixed parameter tractable. There are quite some approximation algorithms for $k$-set packing, the best of which currently achieve an approximation guarantee of $\frac{k+1}{3} + \varepsilon$. On the other hand, the problem cannot be approximated within a factor of $\Omega(\frac{k}{\log k})$. For weighted $k$-set packing there is a $\frac{k+1}{2}$-approximation. There exists both a polynomially sized LP and a polynomially sized SDP for $k$-set packing with integrality gap at most $\frac{k}{3} + 1 + \varepsilon$ as we will show in this paper, where the previous bound on the integrality gap was $\frac{k+1}{2}$.

%\subsection{Applications and related problems}

%\subsection{Current results}

\subsection{Improved integrality gap}\label{subsec:IntroGap}

In this thesis we improve the integrality gap of a linear program of $k$-set packing called the intersecting family LP (see Section \ref{sec:IntersectingLP}).
\begin{theorem}\label{thm:IntegralityGap1}
Let $\varepsilon > 0$. The integrality gap of the intersecting family LP is at most $\frac{k}{3} + 1 + \varepsilon$.
\end{theorem}
Following the results from \cite{LapChiLau}, this immediately implies the following two theorems.
\begin{theorem}\label{thm:LP1}
Let $\varepsilon > 0$. There is a polynomially sized LP for $k$-set packing with integrality gap at most $\frac{k}{3} + 1 + \varepsilon$.
\end{theorem}
\begin{theorem}\label{thm:SDP1}
Let $\varepsilon > 0$. There is a polynomially sized SDP for $k$-set packing with integrality gap at most $\frac{k}{3} + 1 + \varepsilon$.
\end{theorem}
The previous bound on this integrality gap was $\frac{k+1}{2}$ \cite{LapChiLau}. This was a continuation of the work on the standard linear programming relaxation for $k$-set packing. We treat these results on the linear programs in Chapter \ref{chap:LP}. The result on the semidefinite programming relaxation is treated in Chapter \ref{chap:SDP}, along with some background.

\subsection{Simplified proof}\label{subsec:IntroBerman}

We simplify the proof of the main lemma of the currently best weighted approximation algorithm from Berman \cite[Lemma 2]{Berman}. The current proof is very clever but also very algebraic. We make the observation that the squared weight function somehow captures both the maximum weight and the sum of the weights of the neighbourhood of a vertex in the conflict graph of the instance. This allows us to avoid the algebraic proof and simplify it.

We treat this result in Chapter \ref{chap:Weighted}.

\section{Outline of the thesis}\label{sec:Outline}

\paragraph{Chapter 2} Set packing is one of the fundamental optimization problems and therefore there are numerous applications within mathematics and real life. It is highly related to some other well-known problems such as the hypergraph matching problem and the maximum independent set problem. In Chapter \ref{chap:Applications} these applications and related problems are considered.

\paragraph{Chapter 3} In Chapter \ref{chap:Results} the current results on the $k$-set packing problem are discussed. First the unweighted approximation algorithms and the weighted approximation algorithms are considered in Sections \ref{sec:Unweighted} and \ref{sec:Weighted}. For the unweighted problem the best approximation algorithm currently achieves an approximation guarantee of $\frac{k+1}{3} + \varepsilon$ \cite{Cygan,FurerYu} and for the weighted problem the best result is a $\frac{k+1}{2}$-approximation \cite{Berman}.

Then Section \ref{sec:Parameterized} continues with the parameterized algorithms for $k$-set packing, as this problem is fixed parameter tractable. There has been a long line of research in this area and Appendix \ref{app:Parameterized} contains an overview of these algorithms.

Chapter \ref{chap:Results} ends with some results on the inherent hardness of the problem in Section \ref{sec:Hardness}. Subsection \ref{subsec:Hardness0} starts with some hardness results that apply to the general set packing problem. Subsections \ref{subsec:Hardness1}, \ref{subsec:Hardness2} and \ref{subsec:Hardness3} contain theorems on $k$-set packing specifically. It is NP-hard to approximate within a factor of $\Omega\left(\frac{k}{\log k}\right)$ \cite{Hazan} and three other results on the limits of local search techniques for the problem are mentioned \cite{Cygan,FurerYu,Sviridenko}.

\paragraph{Chapter 4} Chapter 4 treats the results on the standard linear programming relaxation and the intersecting family LP and contains the proofs of Theorems \ref{thm:IntegralityGap1} and \ref{thm:LP1}.

\paragraph{Chapter 5} The proof of Theorem \ref{thm:SDP1} is given in Chapter \ref{chap:SDP} along with some background about semidefinite programming and the Lov\'{a}sz Theta function.

\paragraph{Chapter 6} The topic of Chapter \ref{chap:Weighted} is weighted $k$-set packing. Currently the best result is a $\frac{k+1}{2}$-approximation from Berman \cite{Berman}. He provides two algorithms. One is called \textsc{SquareImp} which uses a local search technique using the squared weight function. The second is called \textsc{WishfulThinking} which searches locally for structures which he calls nice claws (definitions are in Chapter \ref{chap:Weighted}). He links these two algorithms in a shrewd way, allowing him to proof both the approximation guarantee and the polynomial running time. At the end we give a simplified proof of the main lemma.

\paragraph{Chapter 7} Finally there is an extended discussion in Chapter \ref{chap:Discussion} about possible improvements. The results on the LP and the SDP relaxation are discussed in Section \ref{sec:DiscLP} and whether these can be extended to the weighted case. We see why changing Berman's \cite{Berman} weighted $\frac{k+1}{2}$-approximation algorithm a bit does not yield an improvement in Section \ref{sec:DiscBerman}. In Section \ref{sec:IS} the strong relation between set packing and the independent set problem (on bounded degree graphs) is considered and it is argued why the results for independent set are better. We discuss why the weighted problem asks for such different algorithms compared to the unweighted version of the problem in Section \ref{sec:WeightedVSUnweighted}. Several suggestions for future research are given. %and why the unweighted $\frac{k+2}{3}$-approximations are so considerably easier than the $\frac{k+1}{3}$-approximations.

\chapter{Applications and related problems}\label{chap:Applications}

\section{Applications}\label{sec:Applications}

Set packing has a lot of applications in capital budgeting, crew scheduling, cutting stock, facilities location, graphs and networks, manufacturing, personnel scheduling, districting, information systems, vehicle routing and timetable scheduling, see \cite{Applications} for a survey. In this section some real life applications are mentioned and in Section \ref{sec:RelatedProblems} the relation of set packing to other combinatorial problems is discussed. Section \ref{sec:LocalSearch} then gives some background on local search as a background for Chapter \ref{chap:Results}.

%\subsection{Latin squares and Sudokus}\label{subsec:LatinSquare}

\paragraph{Latin squares} A nice application of set packing is the extension of partial Latin squares \cite{LatinSquare1,LatinSquare2}. A partial Latin square is an $n \times n$ array in which each cell is either empty or coloured with exactly one of the colours $\{1,\ldots,n\}$. A Latin square is a partial Latin square without empty cells where every colour occurs exactly once in every row and every column. The problem is given a partial Latin square to find a completion that colours as many empty cells as possible such that no rows or columns contain any colour more than once.

%Let the universe $\mathcal{U}$ consist of $n^2$ elements of the form $\{e_i, r_j\}$ for every combination of some colour $i$ and some row $j$, $n^2$ elements of the form $\{e_i, c_k\}$ for every combination of some colour $i$ and some column $k$, and $n^2$ elements of the form $\{r_j, c_k\}$ for every combination of some row $j$ and some column $k$.

This can be modeled as a set packing problem as follows. Let the universe $\mathcal{U}$ consist of $3n^2$ elements of the form $\{e_i, r_j\}$, $\{e_i, c_k\}$ and $\{r_j, c_k\}$ for every combination of some colour $i$, some row $j$ and/or some column $k$. Call an element $\{e_i, r_j\}$ or $\{e_i, c_k\}$ empty if respectively row $j$ or column $k$ does not contain colour $i$, and call $\{r_j, c_k\}$ empty if the cell in row $j$ column $k$ is empty. Now let the collection of subsets $\mathcal{C}$ consist of all triplets $\{\{e_i, r_j\}, \{e_i, c_k\}, \{r_j, c_k\}\}$ where all three elements are empty. This creates a set packing instance where every triplet contained in the solution indicates to colour the cell at row $j$ column $k$ with colour $i$. %By the construction of the subsets, every cell is coloured at most once and every row and column contains each colour at most once.

This idea can be extended to the popular Sudoku puzzles. A Sudoku is a $9 \times 9$ Latin square with the additional constraints that in the nine $3 \times 3$ ``boxes'' every colour is allowed only once. Using quadruples rather than triplets, this problem can be translated to set packing in a similar fashion.
%complete a partial Latin square to a Latin square if possible. This can be modeled as a set packing problem as follows. Let the universe $\mathcal{U}$ consist of $n$ elements for every row (so $n^2$ row-elements), $n$ elements for every column (so $n^2$ column-elements), an element for every combination of a row and a colour (so $n^2$ row-colour-elements) and an element for every combination of a column and a colour (so $n^2$ column-colour-elements). Let the subsets in $\mathcal{C}$ be all quadruples of a row-element, a column element, and the colour-elements corresponding to that row and that column. One now needs to find a perfect packing. Since the subsets need to be chosen mutually disjoint, [TODO] [TODO: Figure out how this works]. The given partial Latin square can be incorporated by enforcing the given cells to contain that colour: for every given cell, delete the corresponding quadruples and elements for example.

%Sudoku: We start with the same sets $\mathcal{U}$ and $\mathcal{C}$ as in the previous setting and add $n^2$ elements of the form $\{e_i, b_l\}$ for every combination of some colour $i$ and some block $l$. This element is called empty when colour $i$ does not occur in box $l$. Now extend every triplet $\{\{e_i, r_j\}, \{e_i, c_k\}, \{r_j, c_k\}\}$ to a quadruple with an extra fourth entry $\{e_i, b_l\}$ if the cell at row $j$ column $k$ is in box $l$. Do not extend the other triplets. The set packing instance with the set of all empty quadruples as $\mathcal{C}$ can now be solved to extend the given Sudoku.

\paragraph{Other applications} Three real life applications one could think of is the assignment of crew members to airplanes \cite{Airplanes}, the generation of a coalition structure in multiagent systems \cite{Coalition} and determining the winners in combinatorial auctions to maximise the profit \cite{Auction1,Auction2,Auction3}.

\section{Related problems}\label{sec:RelatedProblems}

Set packing is highly related to the hypergraph matching problem and the independent set problem. This section surveys these relations and some special cases of set packing.

\subsection{Hypergraph matching}\label{subsec:Hypergraphs}

Set packing and hypergraph matching are really the same problem with different names. This subsection contains some background on hypergraphs and hypergraph matchings and relates these notions to set packing and $k$-set packing.

\paragraph{Hypergraphs} A hypergraph $H$ is a pair $H=(V,E)$ where $V$ is the set of vertices and $E$ is the set of hyperedges. A hyperedge $e \in E$ is a nonempty subset of the vertices. In a weighted hypergraph, every hyperedge $e \in E$ is associated with a weight $w(e)$. A hyperedge $e \in E$ is said to contain, cover or be incident to a vertex $v \in V$ when $v \in e$.

For a vertex $v$, $\delta(v)$ denotes the set of hyperedges incident to it. The degree of a vertex $v$ is $|\delta(v)|$. %We say two vertices $u$ and $v$ are adjacent when there is a hyperedge $e \in E$ containing both $u$ and $v$. There might be multiple hyperedges touching the same set of vertices.
The cardinality of a hyperedge is the number of vertices it contains. When every vertex has the same degree $k$, the hypergraph is called $k$-regular. When every hyperedge has the same cardinality $k$, the hypergraph is $k$-uniform. A graph is thus a 2-uniform hypergraph. In a graph, every edge is a subset of two vertices.

\paragraph{Matchings} A hypergraph matching is a subset of the hyperedges $M \subseteq E$ such that every vertex is covered at most once, i.e. the hyperedges are mutually disjoint. This generalises matchings in graphs. From now on, with matching we mean a hypergraph matching. The cardinality of a matching is the number of hyperedges it contains. %A matching $M$ is called maximal if there is no hyperedge $e \in E \setminus M$ such that $M \cup \{e\}$ is a matching.
A matching is called maximum if it has the largest cardinality of all possible matchings. %A matching is called perfect if every vertex is covered exactly once.
In the hypergraph matching problem, a hypergraph is given and one needs to find the maximum matching.
\begin{quote}
\begin{bf}Hypergraph matching problem\end{bf}\\
Given: a hypergraph $H = (V,E)$ \\
Find: a maximum collection of mutually disjoint hyperedges.
\end{quote}
%
%We call a hypergraph bipartite when the vertices can be partitioned into two sets $U$ and $V$ such that every hyperedge is incident to exactly one vertex from $U$. We say a bipartite matching is perfect if every vertex in $U$ is covered exactly once. %In the Santa Claus problem treated in Subsection \ref{subsec:SantaClaus} we were actually considering a bipartite hypergraph with $U$ as the set of players and $V$ the set of the resources. The problem was really to find a perfect bipartite matching in this hypergraph.
\paragraph{$k$-partite hypergraphs}
The notion of bipartiteness in graphs can be generalised in a hypergraph to the concept of $k$-partiteness. A hypergraph $H = (V,E)$ is called $k$-partite if the set of vertices $V$ can be partitioned into $k$ classes $V_1, \ldots, V_k$ such that every hyperedge $e \in E$ touches exactly one vertex from every vertex class. A $k$-partite hypergraph is $k$-uniform by definition. The hypergraph matching problem on a $k$-partite graph is also called the $k$-dimensional matching problem. This generalises the bipartite matching on graphs where $k=2$ and it is a restricted version of hypergraph matching on a $k$-uniform hypergraph. In particular the 3-dimensional matching problem is well studied in literature. %[TODO] [TODO: add citations].

\paragraph{Relation to set packing} Let a set packing instance $S = (\mathcal{U},\mathcal{C})$ be given. Define a hypergraph $H = (\mathcal{U}, \mathcal{C})$, calling the elements in $\mathcal{U}$ vertices and the subsets in $\mathcal{C}$ hyperedges. Then the set packing problem on $S$ is exactly the same as the hypergraph matching problem on $H$. An instance of the $k$-set packing problem translates to an instance of the hypergraph matching problem on a $k$-uniform hypergraph, as every subset (hyperedge) contains exactly $k$ elements. %This is a generalization of the well-known graph matching where $k=2$. %Even though for the case of $k=2$ the maximum cardinality and even the maximum weight matching can be found in polynomial time, already for $k=3$ the problem becomes NP-hard.

Results for set packing thus immediately apply to hypergraph matching and vice versa.

\subsection{Independent set}\label{subsec:IndependentSet}

\paragraph{Relation to set packing} Set packing is also closely related to the (maximum) independent set problem. Given a graph $G = (V,E)$, an independent set is a subset of vertices that are mutually non-adjacent, i.e. a set of vertices whose induced subgraph does not contain any edge. This subsection treats this problem and its relation to set packing.
\begin{quote}
\begin{bf}Independent set problem\end{bf}\\
Given: a graph $G = (V,E)$ \\
Find: a maximum collection of mutually non-adjacent vertices.
\end{quote}
Let a set packing instance $I = (\mathcal{U},\mathcal{C})$ be given. Create the conflict graph $G = (\mathcal{C},E)$ %with a vertex for every subset in $\mathcal{C}$ and an edge between two vertices $S, T \in \mathcal{C}$ when the two subsets these vertices correspond to are not disjoint (i.e. $S \cap T \neq \emptyset$).
for the instance $I$. Then the edge set $E$ captures all the intersections between the subsets. Any packing of $\mathcal{C}$ now corresponds to an independent set in $G$: as a packing is mutually disjoint by definition, the corresponding vertices are non-adjacent. On the other hand, any independent set in $G$ corresponds to a packing of $\mathcal{C}$. Finding a maximum packing is thus equivalent to finding a maximum independent set in the conflict graph.

\paragraph{Relation to $k$-set packing} Now consider a $k$-set packing instance $(\mathcal{U}, \mathcal{C})$ where every subset in $\mathcal{C}$ contains exactly $k$ elements from $\mathcal{U}$. Consider a fixed subset $S$ and look at the set $N(S)$ of all subsets intersecting $S$. There can be at most $k$ subsets in $N(S)$ that are mutually disjoint, when each subset intersects $S$ in a distinct element. Thus within $N(S)$, the maximum packing has at most cardinality $k$.

Now consider the corresponding conflict graph $G$. $S$ is a vertex in $G$ and $N(S)$ is the set of all neighbours of $S$. By the previous reasoning, $N(S)$ contains an independent set of size at most $k$. Thus in the neighbourhood of any vertex of $G$ there are at most $k$ mutually non-adjacent vertices. In other words, the conflict graph is $k+1$-claw free: it does not contain $K_{1,k+1}$ as an induced subgraph, where $K_{n,m}$ is the complete bipartite graph on $n$ and $m$ vertices. A vertex can have any number of neighbours, but the maximum number of mutually non-adjacent neighbours is $k$. %When the cardinality of every subset is bounded by $k$, the conflict graph in which we need to find an independent set thus inherits this property.

\paragraph{Independent set on bounded degree graphs} On the other hand, the maximum independent set problem in graphs $G$ where all degrees are bounded by $k$ is a special case of $k$-set packing. Map every vertex in $G$ to a subset in $\mathcal{C}$ and map every edge to a distinct element in $\mathcal{U}$. Let a subset in $\mathcal{C}$ (a vertex in $G$) contain all elements in $\mathcal{U}$ (edges in $G$) that are incident to it. Now every subset in $\mathcal{C}$ contains at most $k$ elements because every vertex in $G$ had degree at most $k$. Again, by adding dummy elements to the subsets with less than $k$ elements, every subset in $\mathcal{C}$ has exactly $k$ elements. As such, the $k$-set packing problem generalises the maximum independent set problem on bounded degree graphs.

Section \ref{sec:IS} discusses the relation between the results on set packing and independent set more thoroughly.

\subsection{Special cases of set packing}

\paragraph{Clique and triangle packing} There are some special cases of set packing we would like to mention. Instead of studying the set packing problem with a bound on the cardinality of every subset, it is also possible to study the problem with certain given structures on the given sets. One example is the clique packing problem, where one is given some graph $G$. A packing of $G$ is a collection of pairwise vertex-disjoint subgraphs of $G$, each of which is isomorphic to a clique. A packing is said to cover an edge of $G$ if one of the subgraphs contains that edge. In the clique packing problem, one tries to maximise the number of covered edges \cite{CliquePacking}. The special case where every clique is of size 3 is called the triangle packing problem \cite{TrianglePacking2,TrianglePacking1}. The set packing problem generalises these problems.

%\subsection{Tree-like weighted set packing}\label{subsec:TreeLike}

\paragraph{Tree-like weighted set packing} The tree-like weighted set packing problem is another subproblem with an additional structure on the sets that allows for better results that even extend to the weighted version. In this problem, the subsets in $\mathcal{C}$ can be organised into a forest $F$ satisfying the following properties. Every vertex in $F$ corresponds to one subset in $\mathcal{C}$. If a vertex $Y$ is a child of a vertex $X$, then $Y$ is a subset of $X$. If $Y$ and $Z$ are distinct children from $X$, then $Y$ and $Z$ are disjoint. And the roots of all the trees are pairwise disjoint. When the sets are structured like this, the problem can be solved exactly in cubic time using a dynamic programming algorithm \cite{TreeSP}.

\paragraph{$k$-dimensional matching} Perhaps the best known special case of $k$set packing is the $k$-dimensional matching problem. As was mentioned in Subsection \ref{subsec:Hypergraphs}, this is the hypergraph matching problem on a $k$-partite hypergraph, which is a stronger structure than a $k$-uniform hypergraph. In particular 3-dimensional matching is well-known. All results for $k$-set packing immediately apply to $k$-dimensional matching. %and there are some results specifically for $k$-dimensional matching that exploit this extra structure.

%Some subproblems with additional structure on either the structure or the weights of the sets allow for better results. One example is the tree-like weighted set packing problem. In this problem, the subsets in $\mathcal{C}$ can be organized into a forest $F$ satisfying the following properties. Every vertex in $F$ corresponds to one subset in $\mathcal{C}$. If a vertex $Y$ is a child of a vertex $X$, then $Y$ is a subset of $X$. If $Y$ and $Z$ are distinct children from $X$, then $Y$ and $Z$ are disjoint. And the roots of all the trees are pairwise disjoint. When the sets are structured like this, the problem can be solved exactly in cubic time using a dynamic programming algorithm \cite{TreeSP}.

%The weighted set packing problem can also be approximated better when the weights of the set are structured like the following: every element $u$ in $\mathcal{U}$ is associated with a weight $w(u)$ and the weight of any subset in $\mathcal{C}$ is the sum of the weights of the elements it contains. In this case, there is an approximation algorithm for this problem within a factor of $O(\frac{n}{\log n})$ \cite{SpecialWeight}. [TODO] [TODO: This is not true, IS on hypergraphs is defined differently in this paper so other problem?!]

\section{Local search}\label{sec:LocalSearch}

\paragraph{Background} A lot of the approximation algorithms for $k$-set packing are related to the notion of local search techniques. In this section we survey this related topic. Local search is a natural heuristic to tackle difficult problems. The most basic solution approaches for a discrete optimization problem are generating just one solution (e.g. using a priority-based heuristic), enumerating implicitly all possible solutions (e.g. with branch and bound techniques) and generating several solutions and choosing the best. Local search is a heuristic that finds from an initial solution a sequence of solutions that use the previous solution for the next solution, with the goal to increase the objective value from the initial solution to a better value. A neighbourhood structure is defined on the solutions that somehow resembles similar solutions. Until some stopping criterion is reached, in every iteration a solution from the neighbourhood of the current solution is chosen.

A local search heuristic usually consists of four main elements: a method to calculate an initial solution, a definition of the neighbourhood of a solution, a criterion to choose a solution from any neighbourhood, and a stopping criterion. The method to calculate an initial solution is needed to start a local search, but is usually not seen as a part of the local search technique itself. For example, Chapter \ref{chap:Results} contains a lot of results stating that a locally optimal solution with respect to some neighbourhood structure achieves a certain approximation guarantee, without going into the details of how to find an initial solution or sometimes even how to iterate to the next solution.

\paragraph{Results} Local search techniques are widely used in practice, the most common example being for the traveling salesman problem. However, there are not much positive results on local search in theory. In fact, a complexity class called $\mathcal{PLS}$ (polynomial time local search) has been defined \cite{PLS}. Problems that are $\mathcal{PLS}$-complete are as hard as the hardest local optimization problems. A lot of well-known heuristics are known to be $\mathcal{PLS}$-complete. There are even some problems for which it has been proved that there is no sequence, of less than exponential length, of improvements ending in a locally optimal solution \cite{Halldorsson}. To make matters worse, there are popular heuristics with good average-case behaviour that perform quite bad at relatively easy problems. For instance, the algorithm at the heart of the well-known simulated annealing performs quite poor at finding a maximum matching (which is even in $\mathcal{P}$).

On the positive side, set packing does not fit this negative theoretic framework. Currently all approximation algorithms for both the weighted and the unweighted $k$-set packing problem use local search at their core. %In \cite{Halldorsson} a lot of results are mentioned on different problems, all using local search.
However, there are also some results on the limits of the approximation guarantees using local search.
The next chapter proceeds with approximation algorithms for $k$-set packing %, which are all closely related to the framework of local search,
and the bounds on what is achievable using this technique. %To make sure no confusion may arise we first define some notation and definitions. 
\chapter{Current results}\label{chap:Results}

\section{Unweighted approximations}\label{sec:Unweighted}

In this section we briefly mention the state of the art of unweighted approximations for the $k$-set packing problem. There has been a long line of work in this area and Subsection \ref{subsec:HurkensSchrijver} starts with the first $\frac{k}{2} + \varepsilon$-approximation given already in 1989 by Hurkens and Schrijver \cite{HurkensSchrijver}. We give a proof and some intuition for the case of a 2-locally optimal solution.

Subsection \ref{subsec:UnweightedQuasi} continues with a $\frac{k+2}{3}$-approximation \cite{Halldorsson} and a $\frac{k+1}{3} + \varepsilon$-approximation \cite{Mastrolilli}. These search for improving sets of size $O(\log n)$, so the improved approximation guarantee is at the cost of the running time, which is $n^{O(\log n)}$ for both algorithms (also called quasi-polynomial running time).

%These are local search algorithms that search for improving sets of size $O(\log n)$. Their clever insight is that such a search space allows them to improve the approximation guarantee from about $\frac{k}{2}$ to about $\frac{k}{3}$. This is at the cost of the running time though: searching a $O(\log n)$ space costs quasipolynomial time, i.e. $O(n^{\log n})$ time.

Only this year these algorithms were adapted to run in polynomial time \cite{Cygan,Sviridenko}. These results are mentioned in Subsection \ref{subsec:UnweightedPoly}, together with an adaptation of the $\frac{k+2}{3}$-approximation to turn it into another $\frac{k+1}{3}$-approximation \cite{FurerYu}. These are currently the best results.

Next, Section \ref{sec:Weighted} treats the current state of the art of the weighted approximation algorithms, followed by the parameterized algorithms and the hardness results in Sections \ref{sec:Parameterized} and \ref{sec:Hardness}. The results on the linear programming and semidefinite programming relaxations are postponed to the next chapters, as these will be treated in more detail and we will provide a new bound for their integrality gap.

\subsection{The first approximation algorithm}\label{subsec:HurkensSchrijver}

%Here is a brief overview of the current results in the area of unweighted approximations for the $k$-set packing problem. %For a maximization problem, an approximation algorithm is said to approximate the problem within a factor $\rho > 1$ when the output $O$ of the algorithm satisfies $\frac{
The unweighted $k$-set packing problem can be solved in polynomial time for $k=2$ \cite{Minty}, so from now on assume $k \geq 3$.

\paragraph{2-locally optimal solution} Perhaps the easiest local search technique to try is to iteratively search for an improving set of size 2. This is either one subset in $\mathcal{C}$ that does not intersect the current solution $\mathcal{A}$ or a pair of subsets that intersect at most one subset of $\mathcal{A}$. Starting with the empty solution and iteratively adding improving sets of size 2 yields an algorithm that is $\frac{k+1}{2}$-approximate \cite{Weighted2,Halldorsson}.

\paragraph{First approximation algorithm} A natural extension of this search technique is to search for improving sets of larger cardinality. If instead of improving sets of size 2 improving sets of constant size $s$ are searched for increasing values of $s$, it is possible to obtain a polynomial time approximation ratio of $\frac{k}{2} + \varepsilon$. This was discovered by Hurkens and Schrijver \cite{HurkensSchrijver} and this was actually the first approximation algorithm for $k$-set packing. %For every $\varepsilon$ there exists an $s$ such that an $s$-locally optimal solution has approximation guarantee $\frac{k}{2} + \varepsilon$.
Even though the problem was well-studied in the years that followed, this remained the best polynomial time approximation for over 20 years.

To be precise, they proved the following.
\begin{theorem}\label{thm:HurkensSchrijver}
(\cite[Theorem 1]{HurkensSchrijver}) Let $E_1, \ldots, E_m$ be subsets of the set $V$ of size $n$, such that:
\begin{enumerate}
  \item Each element of $V$ is contained in at most $k$ of the sets $E_1, \ldots, E_m$;
  \item Any collection of at most $t$ sets among $E_1, \ldots, E_m$ has a system of distinct representatives\footnote{A system of distinct representatives of sets $\{E_i\}$ is a set $\{e_i\}$ such that $e_i \in E_i$ for all $i$ and $e_i \neq e_j$ for $i \neq j$.}.
\end{enumerate}
Then, we have the following:
\begin{equation*}
\begin{alignedat}{2}
\frac{m}{n} \leq \frac{ k (k-1)^r - k }{ 2 (k-1)^r - k } \quad & \textrm{ if } t = 2r-1; \ \\
\frac{m}{n} \leq \frac{ k (k-1)^r - 2 }{ 2 (k-1)^r - 2 } \quad & \textrm{ if } t = 2r.
\end{alignedat}
\end{equation*}
\end{theorem}
\paragraph{Proof for $t=2$} They provide a very keen yet complicated proof. We will give the proof and some intuition for the case $t=2$, which establishes that any 2-locally optimal solution is at most a factor of $\frac{k+1}{2}$ away from the optimal solution. Let $\mathcal{A}$ be any 2-locally optimal solution and $\mathcal{B}$ be any optimal solution. Denote by $\mathcal{B}_1$ and $\mathcal{B}_2$ the sets in $\mathcal{B}$ that intersect $\mathcal{A}$ in one set respectively at least two sets. As $\mathcal{A}$ is 2-locally optimal there are no sets in $\mathcal{B}$ that do not intersect any set of $\mathcal{A}$. Note that every set in $\mathcal{A}$ intersects with at most $k$ sets in $\mathcal{B}$. So
\begin{equation*}
|\mathcal{B}_1| + 2|\mathcal{B}_2| \leq k|\mathcal{A}|.
\end{equation*}
Since $\mathcal{A}$ is 2-locally optimal, $|\mathcal{B}_1| \leq |\mathcal{A}|$ and hence
\begin{equation*}
2|\mathcal{B}| = |\mathcal{B}_1| + 2|\mathcal{B}_2| + |\mathcal{B}_1| \leq k|\mathcal{A}| + |\mathcal{A}| = (k+1)|\mathcal{A}|,
\end{equation*}
and therefore $\frac{|\mathcal{B}|}{|\mathcal{A}|} \leq \frac{k+1}{2}$.

%The proof is not hard and we will give the proof and some intuition for the case $t=2$, which shows that any 2-local optimal solution is at most a factor of $\frac{k+1}{2}$ away from the optimal solution. Let us denote the singletons to be $\{E_1, \ldots, E_l\}$ for some $l \leq m$. Because $t=2$ we know that all singletons are distinct, implying $l \leq n$. Now note that
%%
%\begin{equation*}
%2m - l = l + 2 ( m - l ) \leq \sum_{i=1}^l |E_i| + \sum_{i=l+1}^m |E_i| = \sum_{i=1}^m |E_i| \leq kn.
%\end{equation*}
%%
%Therefore we have
%%
%\begin{equation*}
%2m = 2m - l + l \leq kn + n = (k+1)n,
%\end{equation*}
%%
%and we see that $\frac{m}{n} \leq \frac{k+1}{2}$.

\paragraph{Intuition} Intuitively the argument boils down to the following. Look at the intersection graph of the current 2-locally optimal solution $\mathcal{A}$ and some optimal packing $\mathcal{B}$. The sets in $\mathcal{B}$ that conflict with just one set in $\mathcal{A}$ are not particularly interesting as they do not form an improving set. Let us delete these sets from the intersection graph together with the sets in $\mathcal{A}$ they intersect. Because $\mathcal{A}$ is a 2-locally optimal solution, the sets in $\mathcal{B}$ conflicting with two sets in $\mathcal{A}$ still have degree at least 1, while the degrees of the other sets can be anything. That is fine, because in general when every degree is at most two we are done: the number of edges between $\mathcal{A}$ and $\mathcal{B}$ is then at most $2m$ and this number is upper bounded by $nk$ and hence $2m \leq nk$. So the core of the argument really consists of small conflicts posing no problem and concentrating on the remainder of the graph. The iterative process of deleting these sets cannot go on forever, because there are only so many vertices in $\mathcal{A}$ and possibly much more in $\mathcal{B}$.

\subsection{Quasi-polynomial time algorithms}\label{subsec:UnweightedQuasi}

%These are local search algorithms that search for improving sets of size $O(\log n)$. Their clever insight is that such a search space allows them to improve the approximation guarantee from about $\frac{k}{2}$ to about $\frac{k}{3}$. This is at the cost of the running time though: searching a $O(\log n)$ space costs quasipolynomial time, i.e. $O(n^{\log n})$ time.

\paragraph{A $\frac{k+2}{3}$-approximation} Halld\'{o}rsson \cite{Halldorsson} was the first to find $\frac{k+2}{3}$-approximate solutions. His clever insight is that using a search space of improving sets of size $O(\log n)$ allows to improve the approximation guarantee from about $\frac{k}{2}$ to about $\frac{k}{3}$. However, in a straightforward implementation of an iterative algorithm with such a search space, the running time would be quasi-polynomial, i.e. $n^{O(\log n)}$. The deduction of this running time is straightforward: one can find an improving set of size $O(\log n)$ in time $n^{O(\log n)}$, and the number of iterations is trivially upper bounded by $n$.

\paragraph{A $\frac{k+1}{3} + \varepsilon$-approximation} This year, Cygan, Grandoni and Mastrolilli \cite{Mastrolilli} improved upon this result and showed that the approximation guarantee of an $O(\log n)$-locally optimal solution can in fact be bounded by $\frac{k+1}{3} + \varepsilon$. The running time is still quasi-polynomial. Their analysis is more involved and they use a lemma from \cite{BermanMIS} as a black box. In Chapter \ref{chap:LP} the same lemma is used to prove an improved bound on the integrality gap of an LP formulation for the $k$-set packing problem.

\subsection{Polynomial time algorithms}\label{subsec:UnweightedPoly}

%Also continuing the work of Halld\'{o}rsson, Sviridenko and Ward \cite{Sviridenko} combined his analysis with the famous color coding technique \cite{ColorCoding} and an application of dynamic programming. This is a special case of the more general framework of searching for improving sets of size $O(\log n)$, where they use results from fixed parameter tractability to find well-structured improving sets of size $O(\log n)$, but now in polynomial time. This is the first polynomial time approximation since \cite{HurkensSchrijver}. The approximation guarantee is $\frac{k+2}{3}$.

\paragraph{A $\frac{k+2}{3}$-approximation} Sviridenko and Ward \cite{Sviridenko} recently established an elegant polynomial time approximation algorithm that beats the approximation guarantee of $\frac{k}{2}$ and this is the first polynomial time improvement over the $\frac{k}{2} + \varepsilon$ result from \cite{HurkensSchrijver}. Their approximation guarantee is $\frac{k+2}{3}$.

Their key insight is that using improving sets with a special structure in the analysis of Halld\'{o}rsson \cite{Halldorsson} allows to reduce the running time. They look at a graph with a vertex for every set in the current solution $\mathcal{A}$. Two vertices $S,T$ are connected if there is a set in $\mathcal{C} \setminus \mathcal{A}$ that intersects $\mathcal{A}$ only in $S$ and $T$. Possibly $S = T$, which gives rise to a loop; so every set that intersects with at most two sets in $\mathcal{A}$ corresponds to one edge. Then they define three canonical improvements, which are three types of connected graphs containing two distinct cycles, thus having more edges than vertices. They iteratively search for canonical improvements, and when one is found they remove the vertices from $\mathcal{A}$ and add the edges from $\mathcal{C} \setminus \mathcal{A}$ to the current solution.

Another key insight is then to combine this with the famous color-coding technique \cite{ColorCoding} and an application of dynamic programming, which allows them to efficiently search for these canonical improvements. The result is a new polynomial time approximation after more than 20 years.

\paragraph{A $\frac{k+1}{3} + \varepsilon$-approximation} Only two months later the polynomial time approximation guarantee was cleverly improved to $\frac{k+1}{3} + \varepsilon$ by Cygan \cite{Cygan}. He adapted the quasi-polynomial $\frac{k+1}{3}$-approximation from \cite{Mastrolilli}, also with the sharp insight that structured sets reduce the search space and the running time without losing the good approximation guarantee. His algorithm searches for elegant structures that he calls improving sets of bounded pathwidth. Again using the color-coding technique, the algorithm can search the space of $O(\log n)$ improving sets of bounded pathwidth in polynomial time.

%This is a continuation of the work in \cite{Mastrolilli}, which was a quasipolynomial time algorithm with the same approximation guarantee. He managed to turn this into a polynomial time algorithm by only considering what he calls improving sets of bounded pathwidth. Also using the color coding technique, he can search the space of $O(\log n)$ improving sets of bounded pathwidth in polynomial time.

\paragraph{Another $\frac{k+1}{3} + \varepsilon$-approximation} Independently, another $\frac{k+1}{3} + \varepsilon$ polynomial time approximation algorithm was found four months later by F\"{u}rer and Yu \cite{FurerYu}. This is an ingenious improvement of the polynomial time $\frac{k+2}{3}$-approximation of \cite{Sviridenko}. Let $\mathcal{A}$ denote the current solution again. Define a bipartite auxiliary graph $G_\mathcal{A}$ with a vertex for every set in $\mathcal{A}$ in one colour class and a vertex for every set in $\mathcal{C} \setminus \mathcal{A}$ in the other colour class. Connect two vertices from both colour classes if and only if the sets corresponding to them intersect. Indeed, this is the conflict graph of the instance with only a subset of the edges.

The algorithm starts by looking for improving sets of constant size as in \cite{HurkensSchrijver}. Then they guess $O(\log n)$ sets $\mathcal{I}$ in $\mathcal{C} \setminus \mathcal{A}$ that might form an improving set. Now consider any collection of sets $\mathcal{I}_3$ within $\mathcal{I}$ of degree at least 3 in the graph $G_\mathcal{A}$. For such a collection, the algorithm looks for a sequence of replacements that swaps $t$ sets in $\mathcal{A}$ with $t$ sets outside $\mathcal{A}$. The idea is to increase the number of sets in $\mathcal{C} \setminus \mathcal{A}$ that intersect $\mathcal{A}$ in at most two sets. If the degree of the vertices in $\mathcal{I}_3$ now drops to 2 or less, they check whether these vertices together with sets of degree at most 2 in $\mathcal{I}$, form a canonical improvement as in \cite{Sviridenko}. If so, this canonical improvement is added to obtain a solution of larger cardinality.

Again the color-coding technique is exploited to find all structures in polynomial time. In comparison, the local improvements in this paper are less general than in the other $\frac{k+1}{3} + \varepsilon$ polynomial time approximation by Cygan. This results in the fact that this new analysis is simpler than the one in \cite{Cygan}.

%The smart observation is that the improving sets the algorithm searches for can be less structured than in \cite{Sviridenko}, which allows for an improved approximation guarantee. However, the local improvements are more structured than in \cite{Cygan}. This results in the fact that the new analysis in \cite{FurerYu} is simpler than the one in \cite{Cygan}.

%This is based on the methods of \cite{Sviridenko}. Their algorithm is more general than the local improvements of \cite{Halldorsson} and \cite{Sviridenko} to improve the approximation guarantee. The local improvements in \cite{Cygan} however are more general than in this paper, resulting in the fact that this new analysis is simpler than the one in \cite{Cygan}.

\section{Weighted approximations}\label{sec:Weighted}

This section considers approximation algorithms for weighted $k$-set packing. In the weighted $k$-set packing problem, every set $S$ in $\mathcal{C}$ is associated with a weight $w(S)$ and the objective is to maximise the total weight of the packing rather than its cardinality. %We consider all weighted approximation algorithms %and then we mention some results on special cases.

Also the weighted $k$-set packing problem can be solved in polynomial time for $k=2$ \cite{Minty2}, so from now on assume $k \geq 3$.

\subsection{Three close to $k$-approximations}\label{subsec:Weighted1}

\paragraph{A $k$-approximation} The first approximation algorithm for weighted $k$-set packing is due to Hochbaum \cite{Hochbaum} who achieved an approximation guarantee of $k$. His algorithm first preprocesses the input and then more or less greedily adds sets to the current solution.

\paragraph{A $k - 1 + \frac{1}{k}$-approximation} A weighted $k-1+\frac{1}{k}$-approximation algorithm was found in \cite{Weighted2}. They show that the simple greedy algorithm approximates the weighted problem within a factor of $k$, where the greedy approach successively selects the subset with the highest weight from all subsets that do not intersect any selected subset so far. Using a similar local search analysis as in the unweighted case they are able to improve this bound to $k - 1 + \frac{1}{k}$.

\paragraph{A $k - 1 + \varepsilon$-approximation} This result was slightly improved by Arkin and Hassin \cite{Weighted1}. Their setting is more general in the following sense. They consider a bipartite graph $G = (U \cup V, E)$ where every vertex is associated with a weight. They use the shorthand notation $w(X) = \sum_{x \in X} w(x)$ for a subset of the vertices $X$, and write $E_v$ for the neighbourhood $N(v)$ of some vertex $v$. Then analogously to \cite{HurkensSchrijver} they assume that $|E_v| \leq k$ for all $v \in V$ and that any subset $R \subseteq U$ of at most $t$ nodes satisfies $w(R) \leq w\left( \bigcup_{u \in R} E_r \right)$.

Their main theorem is that $\frac{w(U)}{w(V)} \leq k - 1 + \frac{1}{t}$. The $k - 1 + \frac{1}{k}$-approximation from \cite{Weighted2} is the case where $k = t$, so this result is slightly more general. Effectively, they reached an approximation guarantee of $k - 1 + \varepsilon$.

%Their result is slightly more general and they effectively reached an approximation guarantee of $k - 1 + \varepsilon$.

\subsection{A $\frac{4k+2}{5}$ and a $\frac{2k+2}{3}$-approximation}\label{subsec:Chandra}

\paragraph{The algorithm} Chandra and Halld\'{o}rsson \cite{Chandra} were the first to beat this bound and obtain an approximation guarantee of $\frac{2k+2}{3}$. They combine the greedy algorithm to find an initial solution and local search techniques to improve upon this solution. Essentially they show that either the greedy solution is already good, or the local search improves it quite well. They reduce the problem to the independent set problem in a $k+1$-claw free graph $G$, and their local search finds an improving claw. This is a claw $C$ whose center vertex $v$ is in the current solution $A$ in $G$, and hence the talons $T_C$ of the claw (its other vertices forming an independent set) are not in $A$. They show that using local search to find any improving claw leads to an approximation guarantee of $\frac{4k+2}{5}$, while a local search for the best local improvement is $\frac{2k+2}{3}$-approximate. With the best improvement they mean the claw $C$ with the maximum ratio between the sum of the weights of $T_C$ and the sum of the weights of the neighbours of $T_C$ in $A$.

\paragraph{Time complexity} To show that the algorithm runs in polynomial time, they need to modify it a bit. First they find a solution $A$ using the greedy approach. Then they rescale the weight function such that $w(A) = kn$. Then they keep searching for the best local improvement using the weight function $\lfloor w \rfloor$. Now each iteration increases $\lfloor w \rfloor (A)$ by at least one. Since $\mathcal{A}$ is at most a factor of $k$ away from the optimal solution, the number of iterations is bounded by $k^2n$. Since in every turn they only inspect a polynomial number of candidates, the algorithm now runs in polynomial time.

\subsection{A close to $\frac{2k}{3}$-approximation}\label{subsec:2k/3}

Berman and Krysta \cite{BermanWeighted2} improved the approximation guarantee slightly. For every $k$ they find the optimal value for $\alpha$ such that any 2-locally optimal solution with respect to the ``misdirected'' weight function $w^\alpha$ achieves the best approximation guarantee. Surprisingly there are only three distinct values for $\alpha$ that cover all values of $k$. There is an optimal value of $\alpha$ for $k=3$ which yields an approximation guarantee of $\frac{2k}{3} \approx 0.66667k$, there is an optimal value for $k=4$ which approximates the problem within $\frac{\sqrt{13}-1}{4}k \approx 0.65139k$ and there is a value that is optimal for $k \geq 5$, resulting in a $2^{-\log_3 2}k \approx 0.64576k$-approximation.

\subsection{A $\frac{k+1}{2}$-approximation algorithm}\label{subsec:Berman}

Currently the best approximation algorithm for weighted $k$-set packing is from Berman \cite{Berman}, also using another objective function $w^2$ rather than $w$. Its approximation guarantee is $\frac{k+1}{2}$. Berman defines a function charge$(u,v)$ and searches for minimal claws that satisfy some condition on this charge function, called nice claws. He shows that if this algorithm terminates it achieves the desired approximation guarantee. To show that the algorithm terminates, he shows that every nice claw improves the square of the weight function. Similar to the running time analysis of \cite{Chandra}, he shows that the number of iterations where $w^2$ improves is polynomially bounded. And when there is no more claw that improves $w^2$, there is no nice claw anymore, and hence the algorithm terminates. Therefore it runs in polynomial time and achieves the approximation guarantee $\frac{k+1}{2}$. Chronologically, this result came before the result from the previous subsection \cite{BermanWeighted2}.

Chapter \ref{chap:Weighted} dives into the details of this paper and we show a simplified and more intuitive proof of the fact that a nice claw improves $w^2$.

\section{Parameterized complexity}\label{sec:Parameterized}

\paragraph{Background} Due to the limited results in the area of approximation algorithms for a long time, people started to study the parameterized problem where the cardinality of the solution is assumed to be $m$. In the $k$-set $m$-packing problem the goal is to find $m$ disjoint sets, each of size $k$. The running time of exact algorithms for this problem can be written in the form $f(m)poly(n)$, moving the exponential dependency to the parameter $m$ instead of $n$. There have been huge improvements on the running time of these exact algorithms. This is in spite of the fact that the $k$-set $m$-packing problem is $W[1]$-complete with respect to the parameter $m$ \cite{DowneyFellows} (see Subsection \ref{subsec:Hardness2} for details about $W[1]$-completeness). These results contain %both randomized and deterministic algorithms, and
some algorithms especially designed for the case $k=3$. Table~\ref{tab:3SP} in Appendix \ref{app:Parameterized} gives an overview of the results on the 3-set $m$-packing, and we refer to Table~\ref{tab:kSP} in Appendix \ref{app:Parameterized} for %an overview of the results on the parameterized complexity of
the more general $k$-set $m$-packing problem. Some results also extend to the weighted case, this is mentioned in the last column. We follow the convention from parameterized algorithms to let $O^*(f(m))$ denote $f(m)n^{O(1)}$.

\paragraph{Improvements} Some remarks in the last column of these tables require some explanation. Koutis' \cite{Koutis1} original result was an $O^*(2^{O(m)})$ time algorithm. In \cite{SmallColorCoding,GreedyLocalization} it was pointed out that the constants are huge in this approximation, showing the bound is at least $O^*(32000^{3m})$ when $k=3$. Like many results in this area, Koutis derandomised his algorithm using the color-coding technique from Alon, Yuster and Zwick \cite{ColorCoding}. In \cite{SmallColorCoding,GreedyLocalization} a new perfect hashing technique was introduced, with which Koutis' result could be improved to $O^*(25.6^{mk})$. A similar result is true for the deterministic algorithm by Fellows et al \cite{Fellows}. Its original bound was $exp(O(mk))$, which was actually $O^*((12.7D)^{3m})$ for some $D \geq 10.4$ (when $k=3$). This is improved to $O^*(13.78^{mk})$ using the new perfect hash function.

The first deterministic result of the form $O^*(g(m,k))$ for $k$-set $m$-packing is due to Jia, Zhang and Chen \cite{5.7kk}. Currently the best complexity result (randomised) result for $k$-set $m$-packing is $O^*(f(m,k))$ by Bj\"{o}rklund et al \cite{Bjorklund}. This function is not very readable or insightful, but behaves well for small $k$. For example, when $k=3$, we have $O^*(f(3,m)) \approx O^*(1.493^{3m})$, which is the result mentioned in Table~\ref{tab:3SP}. %It should be noted that
The function is %not $O^*((2-\varepsilon)^{mk})$ for any $\varepsilon$ (but it is
strictly smaller than $O^*(2^{mk})$.%)

\paragraph{Hardness} Some of the fastest parameterized algorithms rely on group algebra theory. %, reducing the problem to the detection of a multilinear polynomial of degree $mk$.
A variable is introduced for every element and a subset is the product of its variables. A packing then corresponds to a multilinear polynomial of degree $mk$. Koutis and Williams \cite{KoutisWilliams} showed that detecting such a polynomial cannot be done in their model in time faster than $O^*(2^{mk})$. These results still hold if the color-coding method is used (\cite{ColorCoding}) or the randomised divide-and-conquer approach (\cite{DivideAndConquer,SmallColorCoding}).

This is the only result on the limit of the time complexity of these parameterized algorithms. There are more hardness results known for the approximation algorithms of non-parameterized $k$-set packing. The next section captures an outline of these hardness results.

\section{Hardness results}\label{sec:Hardness}

Subsection \ref{subsec:Hardness0} considers some results on the general set packing problem. Subsections \ref{subsec:Hardness1}, \ref{subsec:Hardness2} and \ref{subsec:Hardness3} consider hardness results specifically for $k$-set packing. These are respectively a result on the hardness of approximation, on the non-existence of certain algorithms and on the limits of local search techniques for this problem.

\subsection{Hardness of set packing}\label{subsec:Hardness0}

Set packing is one of the standard packing problems, which are closely related to covering problems. Table \ref{tab:PackingCovering} lists some of these problems. Set packing is the most general of these packing problems, together with its LP-dual minimum set cover. We refer to \cite{Reductions} for a survey of these problems.
%Roughly speaking, in a packing problem one needs to find a maximum set such that every element is packed at most once, while in a covering problem one needs to find a minimum set such that every element is covered at least once. In fact, the LP-dual of a packing problem is a covering problem in general and vice versa. Table \ref{tab:PackingCovering} lists some of these LP-dualities. We refer to \cite{Reductions} for a survey of these problems.
%
\begin{table}
\centering
\begin{tabular}{ll}
  \toprule
  \multicolumn{2}{c}{Covering-packing dualities} \\
  \midrule
  Minimum set cover    & Maximum set packing \\
  Minimum vertex cover & Maximum matching \\
  Minimum edge cover   & Maximum independent set \\
  \bottomrule
\end{tabular}
\caption{Every horizontal pair of problems are each other's LP-duals.}
\label{tab:PackingCovering}
\end{table}

%Set packing is the dual of set cover, these two problems are the most general of the problems mentioned here. There are close connections between the problems, and connections with hypergraphs. The edge cover problem and the matching problem are the only problems that can be solved in polynomial time \cite{noPTAS2}.
The following results on the hardness of the problems also apply to the set packing problem. Arora et al showed that, unless $\mathcal{P} = \mathcal{NP}$, the vertex cover does not admit a polynomial time approximation scheme, even on bounded degree graphs \cite{noPTAS1}. Moreover, also assuming that $\mathcal{P} \neq \mathcal{NP}$, it has been shown that there is no constant-ratio polynomial time approximation scheme for independent set \cite{noPTAS2,NoPTAS3}. H{\aa}stad proved that set packing cannot be approximated within $n^{1 - \varepsilon}$, where $n$ is the number of sets, unless $\mathcal{NP} = \mathcal{ZPP}$ \cite{CliqueIsHard}. On the positive side, set packing can be approximated within a factor of $\sqrt{N}$ \cite{GeneralSP} (recall that $N$ is the number of elements in $\mathcal{U}$). H{\aa}stad's result also implies this is the best possible assuming $\mathcal{NP} \neq \mathcal{ZPP}$. %\cite{CliqueIsHard}.

\subsection{Hardness of approximation}\label{subsec:Hardness1}

As noted in Subsection \ref{subsec:UnweightedPoly}, currently the best approximation guarantee for unweighted $k$-set packing is $\frac{k+1}{3} + \varepsilon$. Hazan, Safra and Schwarz \cite{Hazan} showed the following hardness of approximation result. There is still a gap to bridge between the currently best approximation guarantee and this bound.

%Even though there has suddenly been an outburst of new polynomial time approximation algorithms for the unweighted case, there is still quite a gap to bridge between the current bound of $\frac{k+1}{3} + \varepsilon$ and the following hardness of approximation result of Hazan, Safra and Schwarz \cite{Hazan}.
%
\begin{theorem}\label{thm:HardnessHazan}
(\cite{Hazan}) It is NP-hard to approximate $k$-set packing in polynomial time within a factor of $O \left( \frac{k}{\log k} \right)$.
\end{theorem}
Here is a rough outline of their argument. Define gap-$P$-$[a,b]$ to be the following decision problem: decide on an instance of the decision problem $P$ whether there exists a fractional solution of size at least $b$ or whether every solution of the given instance is of fractional size smaller than $a$. Then if gap-$P$-$[a,b]$ is NP-hard, so is approximating $P$ within a factor smaller than $\frac{b}{a}$.

Define MAX-3-LIN-$q$ as the optimization problem where a set of linear equations over $GF(q)$ is given, each depending on 3 variables, and one needs to find an assignment maximising the number of satisfied equations. H{\aa}stad \cite{Hardness1} proved that gap-MAX-3-LIN-$q$-$[\frac{1}{q} + \varepsilon, 1 - \varepsilon]$ is NP-hard, which was a central result in the theory of hardness of approximation. %that formulated the PCP-theorem back in the nineties (see e.g. \cite{PCP1,noPTAS1,PCP2}).

In \cite{Hazan} they provide a polynomial time reduction from MAX-3-LIN-$q$ to $k$-uniform hypergraph matching (i.e. to $k$-set packing). They use what they call a $(q,\delta)$-Hyper-Graph-Edge-Disperser as a gadget (see \cite{Disperser} for background on dispersers and extractors). They construct such a gadget for every variable occurring in the equations. The set of all these gadgets is then the set of vertices of a hypergraph $H$. By a clever construction of the hyperedges they relate the satisfied equations to a packing in this hypergraph. This enables them to show that gap-$k$-SP-$[\frac{4}{q^3} + \varepsilon, \frac{1}{q^2} - 1]$ is NP-hard. As they have $k = \Theta(q \log q)$, they find a bound for the inapproximability factor for $k$-set packing of $\Omega(\frac{k}{\log k})$.

\subsection{Non-existence of certain algorithms}\label{subsec:Hardness2}

%This hardness result holds for any approach one could use to solve $k$-set packing.
There are also some other hardness results known, more specifically on the existence of certain algorithms or on the limits of what is achievable using local search techniques. Cygan \cite{Cygan} proved the following.
\begin{theorem}\label{thm:HardnessCygan}
(\cite[Theorem 1.1]{Cygan}) It is $W[1]$-hard to search the whole space of improving sets of size $r$ efficiently.

More formally, unless $FPT = W[1]$, there is no $f(r)poly(n)$ time algorithm that given a family $\mathcal{C} \subseteq 2^\mathcal{U}$ of sets of size 3 and a disjoint subfamily $\mathcal{A} \subseteq \mathcal{C}$ either finds a bigger disjoint family $\mathcal{B} \subseteq \mathcal{C}$ or verifies that there is no disjoint family $\mathcal{B} \subseteq \mathcal{C}$ such that $|\mathcal{A} \setminus \mathcal{B}| + |\mathcal{B} \setminus \mathcal{A}| \leq r$.
\end{theorem}
\paragraph{Fixed parameter tractability} This theorem requires some explanation (see e.g. \cite{FPT1,DowneyFellows}). $FPT$ is the set of fixed parameter tractable problems, which are problems with input size $n$ that can be solved in time $f(k)n^{O(1)}$ for some function $f$. $FPT$ thus classifies problems according to multiple parameters rather than a single parameter, where it is crucial that functions $f(n,k)$ like $n^k$ are not allowed. If the value of $k$ is fixed, the problem is said to be parameterized, which is the setting of Subsection \ref{sec:Parameterized} with parameter $m$. %There a lot of algorithms whose running time was polynomial in $n$ and exponential in $m$ were mentioned.

\paragraph{$W$-hierarchy of $FPT$} The $W$-hierarchy $\bigcup_{t \geq 0} W[t]$ has been introduced to formalise the level of intractability of problems. We have $FPT = W[0]$ and $W[i] \subseteq W[j]$ whenever $i \leq j$. We do not go into the details and refer the interested reader to \cite{FPT1,DowneyFellows}.

Theorem \ref{thm:HardnessCygan} assumes $FPT \neq W[1]$, which is a widely believed assumption. In particular, this assumption is equivalent to the famous Exponential Time Hypothesis (ETH) \cite{ETH}. If $FPT = W[1]$ then the ETH fails and vice versa \cite{Cygan,FPT1}. When a problem is $W[1]$-complete (i.e. not solvable in polynomial time unless $FPT = W[t]$, here for $t=1$), this is strong evidence that the problem is probably not fixed parameter tractable. For example, %the decision versions of
the clique problem and the independent set problem are $W[1]$-complete and %the decision versions of
the dominating set problem and the set cover problem are $W[2]$-complete. %(indeed, these last two problems are equivalent).

The assumption in Theorem \ref{thm:HardnessCygan} is thus plausible, and if this is true then there is no polynomial time algorithm which searches the whole space of size $r$ improving sets.

\subsection{Limits of local search}\label{subsec:Hardness3}

Next to this result on the existence of such an algorithm, there is also a lower bound on the approximation guarantee that algorithms based on local search techniques can achieve. %This shows the limits what this technique is capable of for $k$-set packing in terms of its approximation guarantee.
Sviridenko and Ward \cite{Sviridenko} proved the following.
\begin{theorem}\label{thm:HardnessSviridenko}
(\cite[Theorem 6.1]{Sviridenko}) The locality gap of a $t$-locally optimal solution is at least $\frac{k}{3}$, even when $t$ is allowed to grow on the order of $n$.

More formally, let $c = \frac{9}{2e^5k}$ and suppose that $t \leq cn$ for all sufficiently large $n$. Then there exist 2 pairwise disjoint collections of $k$-sets $\mathcal{A}$ and $\mathcal{B}$ with $|\mathcal{A}| = 3n$ and $|\mathcal{B}| = kn$ such that any collection of $a \leq t$ sets in $\mathcal{B}$ intersects with at least $a$ sets in $\mathcal{A}$.
\end{theorem}
This result shows that even local search algorithms that are allowed to examine some exponential number of possible improvements at each stage cannot achieve an approximation guarantee better than $\frac{k}{3}$. This suggests that local search algorithms, which currently achieve an approximation guarantee of $\frac{k+1}{3} + \varepsilon$, are not a good approach to beat the approximation guarantee much further. F\"{u}rer and Yu \cite{FurerYu} extended this result. %to the the following.
\begin{theorem}\label{thm:HardnessFurerYu}
(\cite[Theorem 7]{FurerYu}) There is an instance for $k$-set packing with locality gap $\frac{k+1}{3}$ such that there is no local improvement of size up to $O(n^{\frac{1}{5}})$.

More formally, for any $t \leq \left( \frac{3 e^3 n}{k} \right)^\frac{1}{5}$ there exist two disjoint collections of $k$-sets $\mathcal{A}$ and $\mathcal{B}$ with $|\mathcal{A}| = 3n$ and $|\mathcal{B}| = (k+1)n$ such that any collection of $t$ sets in $\mathcal{A}$ intersects with at least $t$ sets in $\mathcal{B}$.
\end{theorem}
So local search has reached its limits for all practical purposes. To achieve an approximation guarantee beating the order of $\frac{k}{3}$, other approaches to the problem are necessary to consider. Another approach one could take is the LP and SDP relaxations for $k$-set packing. The next chapters treats these two approaches and show improved bounds on their integrality gap. 
\chapter{LP formulation}\label{chap:LP}

This chapter treats the current results on the LP relaxation of $k$-set packing. Section \ref{sec:StandardLP} starts with the results for the standard LP relaxation. Then Section \ref{sec:IntersectingLP} shows how to strengthen this LP and we prove a new bound for the integrality gap of this strengthened LP. Finally Section \ref{sec:SizeLP} presents how to decrease the size of this LP to a polynomial size. %Next we will treat some background on semidefinite programming and the Lov\'{a}sz Theta function. Then we will be able to show that there exists a polynomially sized SDP with the same improved integrality gap.

\section{The standard LP}\label{sec:StandardLP}

\paragraph{LP-formulation} View the set packing problem as the hypergraph matching problem like in Subsection \ref{subsec:Hypergraphs}. Introduce a variable $x_e$ for every hyperedge $e$ that indicates whether $e$ is included in the matching or not. The objective is to maximise $\displaystyle \sum_{e \in E} x_e$ such that every vertex is covered only once. For convenience we introduce the following notation.
\begin{equation*}
x(F) := \sum_{e \in F} x_e.
\end{equation*}
Hence the natural linear program looks like the following, denoted by \eqref{LP}. As before, $\delta(v)$ denotes the set of hyperedges incident to $v$.
%
%\begin{equation*}
%\begin{alignedat}{2}
%\text{max}  \quad & x(E) \ \\
%\text{s.t.} \quad & x\left(\delta(v)\right) \leq 1, & \quad & \forall v \in V \\
%                  & x_e \in \{0,1\},                & \quad & \forall e \in E
%\end{alignedat}
%\end{equation*}
%%
%The natural linear program thus becomes the following, denoted by \eqref{LP}.
%%
\begin{equation}\tag{LP}\label{LP}
\begin{alignedat}{2}
\text{max}  \quad & x(E) \ \\
\text{s.t.} \quad & x\left(\delta(v)\right) \leq 1, & \quad & \forall v \in V \\
                  & 0 \leq x_e \leq 1,              & \quad & \forall e \in E
\end{alignedat}
\end{equation}

\paragraph{Results} F\"{u}redi, Kahn and Seymour \cite{Seymour} have shown that the integrality gap of \eqref{LP} equals $k - 1 + \frac{1}{k}$ for $k$-uniform hypergraphs. They also show that in the case of a $k$-partite hypergraph the integrality gap equals $k-1$, but both proofs are non-algorithmic. Chan and Lau \cite{LapChiLau} gave an algorithmic proof of these facts and showed that the bounds are tight. Both results also extend to the weighted case.

\paragraph{Tight example} As a tight example for the non-$k$-partite case, consider the projective plane of order $k-1$. This is a hypergraph that is $k$-uniform (every hyperedge has cardinality $k$), $k$-regular (every vertex has degree $k$), in which every pair of hyperedges intersects in one vertex, and in which every pair of vertices is contained in exactly one hyperedge. Equivalently it is the Steiner system $S(2,k,k^2-k+1)$.
%intersecting (every two hyperedges intersect) and having $k^2-k+1$ hyperedges. %  projective plane of order $k$ is a set of points (vertices) and lines (edges) such that any two points determine a line, any two lines determine a point, every point has $k+1$ lines on it and every line contains $k+1$ points.
A projective plane of order $k$ exists if $k$ is a prime power and the conjecture that this is also a sufficient condition is a long standing open question. The projective plane of order 2 (thus corresponding to 3-set packing) is the well-known Fano plane. Figure \ref{fig:FanoPlane} depicts a Fano plane, where a hyperedge is represented by a line connecting three vertices.

%\begin{figure}
%        \centering
%        \includegraphics[width = 0.3 \textwidth]{Images/FanoPlane.jpg}
%        \caption{The Fano plane. A vertex is represented by a dot and a hyperedge is represented by a line connecting three dots.}
%        \label{fig:FanoPlane}
%\end{figure}

%\begin{figure}
%\centering
%\begin{tikzpicture}[scale=0.5]
%\vertex (a) at (0:0) {};
%\vertex (b) at (0:-10) {};
%\vertex (c) at (0:20) {};
%\vertex (d) at (9:5) {};
%\vertex (e) at (-9:5) {};
%\vertex (f) at (17:-10) {};
%\vertex (g) at (-17:-10) {};
%\path
%(a) edge (b)
%(a) edge (c)
%(a) edge (d)
%(a) edge (e)
%(a) edge (f)
%(a) edge (g)
%(b) edge (d)
%(b) edge (e)
%(b) edge (f)
%(b) edge (g)
%(c) edge (d)
%(c) edge (e)
%(d) edge (e)
%(d) edge (f)
%(d) edge (g)
%(e) edge (f)
%(e) edge (g);
%\end{tikzpicture}
%\caption{The Fano plane. A vertex is represented by a dot and a hyperedge is represented by a line connecting three dots.}
%\label{fig:FanoPlane}
%\end{figure}

\begin{figure}
\centering
\begin{tikzpicture}[main node/.style={circle,fill=black!10,draw,font=\sffamily\Large\bfseries}]
\node[main node] (centre) {\phantom{$_1$}};
\foreach \d/\lbl in {90/A, 210/B, 330/C} {
	\node[main node] (\lbl) at (\d:3.0cm) {\phantom{$_1$}};
}
\foreach \A/\B/\C in {A/B/C, B/C/A, A/C/B} {
	\node[main node] (\A\B) at ($(\A)!0.5!(\B)$) {\phantom{$_1$}};	
}
%\draw[edge] (centre) let \p1 = ($(AB)-(centre)$) in circle({veclen(\x1,\y1)});

\begin{pgfonlayer}{background}
\foreach \A/\B/\C in {A/B/C, B/C/A, A/C/B} {
\draw[edge] (\A) -- (\B);
	\draw[edge] (\A\B) -- (\C);
}
\draw[edge] (centre) let \p1 = ($(AB)-(centre)$) in circle({veclen(\x1,\y1)});
\end{pgfonlayer}

\end{tikzpicture}
\caption{The Fano plane. A hyperedge is represented by a line connecting three vertices.}
\label{fig:FanoPlane}
\end{figure}
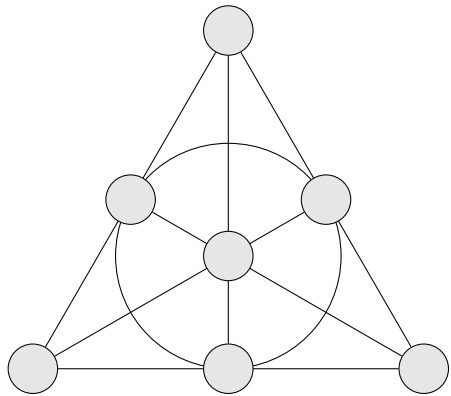

%\begin{figure}
%\centering
%\begin{tikzpicture}
%\node[vertex] (centre) {};
%\foreach \d/\lbl in {90/A, 210/B, 330/C} {
%	\node[vertex] (\lbl) at (\d:0.8cm) {};
%}
%\foreach \A/\B/\C in {A/B/C, B/C/A, A/C/B} {
%	\node[vertex] (\A\B) at ($(\A)!0.5!(\B)$) {};
%	\draw (\A) -- (\B);
%	\draw (\A\B) -- (\C);
%}
%\draw (centre) let \p1 = ($(AB)-(centre)$) in circle({veclen(\x1,\y1)});
%\end{tikzpicture}
%\caption{The Fano plane. A vertex is represented by a dot and a hyperedge is represented by a line connecting three dots.}
%\label{fig:FanoPlane}
%\end{figure}

To see that the projective plane of order $k-1$ is a tight example, note that the integral solution to this $k$-set packing instance equals 1 as every hyperedge intersects every other hyperedge. But fractionally, it is possible to set $x_e = \frac{1}{k}$ for every hyperedge because the hypergraph is $k$-regular. This is a feasible solution to \eqref{LP} and since the hypergraph has $k^2-k+1$ hyperedges the integrality gap equals $\frac{1}{k}(k^2-k+1) = k - 1 + \frac{1}{k}$.

\section{Strengthening the LP formulation}\label{sec:IntersectingLP}

\subsection{The intersecting family LP}\label{subsec:IntersectingFamilyLP}

The LP formulation can be strengthened by adding extra local constraints. Call a family of hyperedges intersecting if every two of them overlap in at least one vertex. In the conflict graph of an instance, an intersecting family $\mathcal{F}$ would be a clique $F$. Obviously, from every intersecting family only one hyperedge can be picked. For every intersecting family thus an extra constraint can be added to obtain the following strengthened LP. Let $\mathcal{K}$ denote the collection of all intersecting families.
\begin{equation}\tag{Intersecting family LP}
\begin{alignedat}{2}
\text{max}  \quad & x(E) \ \\
\text{s.t.} \quad & x\left(\delta(v)\right) \leq 1, & \quad & \forall v \in V \\
            \quad & x\left(K\right) \leq 1,         & \quad & \forall K \in \mathcal{K} \\
                  & 0 \leq x_e \leq 1,              & \quad & \forall e \in E
\end{alignedat}
\end{equation}
This is called the interesting family LP. Note that this LP in general has exponentially many constraints and hence it is not solvable in polynomial time. In Section \ref{sec:SizeLP} we will show that for each $k$ the LP can be rewritten into an LP with a number of constraints that is polynomial in $n$. In \cite{LapChiLau}, they proved the following theorem about this intersecting family LP.
\begin{theorem}\label{thm:LPk+1/2}
(\cite[Theorem 4.1]{LapChiLau}) The ratio between any LP solution to the intersecting family LP and any 2-locally optimal solution is at most $\frac{k+1}{2}$. Thus the integrality gap of the intersecting family LP is at most $\frac{k+1}{2}$.
\end{theorem}
We omit the proof, because we will show an improved bound in Theorem \ref{thm:IntegralityGap2}. In order to proof this, the lemmas from \cite{BermanMIS} treated in the next subsections are necessary.

\subsection{A lemma on multigraphs}\label{subsec:lemma1}

\begin{lemma}\label{lem:BermanMIS1}
(\cite[Lemma 3.1]{BermanMIS}) Assume that every vertex in a multigraph\footnote{A multigraph is a graph where multiple edges and loops are allowed.} $G = (V,E)$ has degree at least 3. Then every vertex $v \in V$ belongs to a connected induced subgraph $G[X]$ with strictly more edges than vertices, of at most $4 \log_2 n - 1$ vertices.
\end{lemma}
\begin{proof}
This is a slightly more detailed proof than in \cite{BermanMIS}. Let $G = (V,E)$ be a multigraph where every vertex has degree at least 3 and let $v \in V$ be an arbitrary vertex. Consider a breadth-first search tree $T$ of $G$ rooted at $v$. The distance of a vertex in $T$ is the length of the shortest path from the vertex to $v$.

Let $m > 0$ and suppose every vertex in $T$ at a distance less than $m$ has at least two children. Then $T$ has at least $2^m$ vertices. Since $T$ has $n$ vertices, $m \leq \log_2 n$. Since $v$ has degree at least 3, in fact $m < \log_2 n$. This implies that there must be a vertex $u$ at a distance of at most $\log_2 n - 1$ having at most one child.

Since $u$ has at least degree 3 but at most one child, one of the edges incident to $u$ in $T$ must be a cross edge, a loop or a multiple edge. Let's denote this edge by $e = \{u,w\}$, where possibly $w = u$. Then the tree paths from $v$ to $u$ and from $v$ to $w$ together with the edge $e$ form an induced subgraph $G[Y]$ of at least as many edges as vertices, and this induced subgraph has at most $2 \log_2 n$ vertices.

If the number of edges in $G[Y]$ is already strictly larger than the number of vertices in $G[Y]$, then the proof is finished. Otherwise the number of edges equals the number of vertices. In this case, shrink $G[Y]$ to one single vertex $y$. In this modified graph $G'$, every vertex still has degree at least 3. Therefore $G'$ contains another induced subgraph induced by some $Z\cup\{y\}$ with at least as many edges as vertices, of size at most $2 \log_2 n$. But then $G[Y] \cup G[Z]$ contains strictly more edges than vertices in $G$ and its size is at most $4 \log_2 n - 1$, because $y$ is not a vertex in $G$. So $G[Y \cup Z]$ is an induced subgraph of $G$ that satisfies the required properties, as $G[Y \cup Z]$ has at least as many edges as $G[Y] \cup G[Z]$. %Note that $G[Y] \cup G[Z]$ has strictly more edges than vertices because the pseudo-vertex $y$ does not count anymore.
\end{proof}

\subsection{Another lemma on multigraphs}\label{subsec:lemma2}

Using this lemma following lemma can be proved, which is the one really needed for Theorem \ref{thm:IntegralityGap2}.
\begin{lemma}\label{lem:BermanMIS2}
(\cite[Lemma 3.2]{BermanMIS}) For any integer $h \geq 1$, any undirected multigraph $G = (V,E)$ with $|E| \geq \frac{h+1}{h} |V|$ contains a set $X$ of less than $4 h \log_2 n$ vertices, such that in $G[X]$ there are more edges than vertices.
\end{lemma}
\begin{proof}
Also this is a slightly more detailed proof than in \cite{BermanMIS}. Let $h$ be any positive integer. For convenience, for any undirected multigraph $H = (V(H),E(H))$ denote $|V(H)| = V_H$, $|E(H)| = E_H$. $H$ is said to satisfy Property $(*)$ if $E_H \geq \frac{h+1}{h} V_H$. Given a set of vertices $U$ in a multigraph $H$, denote $V_U := V_{H[U]}$, $E_U := E_{H[U]}$. $U$ is said to satisfy Property $(*)$ if its induced subgraph $H[U]$ satisfies Property $(*)$.

Let $G = (V,E)$ be an undirected multigraph that satisfies Property $(*)$, so $|E| \geq \frac{h+1}{h} |V|$. We need to show that $T$ contains a set $X$ of less than $4h \log_2 n$ vertices such that in $G[X]$ there are more edges than vertices.

Let $U \subseteq V$ be the smallest set of vertices in $G$ that satisfies Property $(*)$. Because $U$ is a minimal set, $G[U]$ cannot have a vertex $u$ of degree 1: otherwise $U \setminus \{u\}$ would be a smaller set satisfying Property $(*)$.

Now consider a maximal chain of degree 2 vertices $C$ in $G[U]$ and denote the two vertices adjacent to $C$, its endpoints in $G[U] \setminus C$, by $x$ and $y$. First we show that $C$ has less than $h$ vertices. Assume the contrary and suppose $C$ has exactly $h$ vertices. We claim that $U \setminus C$ is a smaller set of vertices that satisfies Property $(*)$ which is a contradiction. Note that $E_{U \setminus C} = E_U - h - 1$ because $G[C]$ contains $h-1$ edges and is connected to the rest of $U$ by 2 other edges. Also $V_{U \setminus C} = V_U - h$. So we need to show that $E_U - h - 1 \geq \frac{h+1}{h} (V_U - h)$. But this is true: subtract $h+1$ on both sides from the known fact that $E_U \geq \frac{h+1}{h} V_U$ to get to this equation.
%To show this, start with the given fact that $E_U \geq \frac{h+1}{h} V_U$. Now subtract $h+1$ on both sides to obtain the equivalent inequality $E_U - h - 1 \geq \frac{h+1}{h} V_U - (h+1)$. We write $h+1 = \frac{h+1}{h} \cdot h$ for the last term to write $E_U - h - 1 \geq \frac{h+1}{h} V_U - \frac{h+1}{h} \cdot h$. We can write this as $E_U - h - 1 \geq \frac{h+1}{h} (V_U - h)$, which is what we wanted to show.
So if $C$ contains $h$ vertices, $U$ is not the smallest set of vertices satisfying Property $(*)$, which is a contradiction. Note that the argument still holds if $C$ has $p > h$ vertices: then on the left-hand side $p+1$ is subtracted while on the right-hand side $\frac{h+1}{h}p = p + \frac{p}{h} > p+1$ is subtracted, so the left-hand side is indeed still larger than the right-hand side. We conclude that $C$ has strictly less than $h$ vertices.

Now modify the graph $G[U]$. Replace every maximal chain of degree 2 vertices $C$ by a single edge connecting its endpoints $x$ and $y$. Since any two such chains do not intersect (except possibly at their endpoints), every such chain is replaced independently from another. Denote the resulting graph, obtained from $G[U]$ by contracting all such chains, by $G'$. Then $G'$ does not have any vertices of degree 2 as they are all contracted.

%Now consider a chain of degree 2 nodes $\{v_1, \ldots, v_p\}$ in $G[U]$ and denote its endpoints by $x$ and $y$ (so $x$ is adjacent to $v_1$ and $y$ is adjacent to $v_p$). Any such chain of degree 2 nodes, together with its endpoints, has $p+1$ edges and $p+2$ vertices, and hence less edges than vertices. We can therefore replace every maximal such chain with a single edge connecting $x$ and $y$. Since any two such maximal chains do not intersect (except possibly at their endpoints), we can replace every maximal such chain independently from another. Denote the resulting graph, obtained from $G[U]$ by these contractions, by $G'$. Then $G'$ does not have any vertices of degree 2.

Therefore all vertices in $G'$ have degree at least 3, and Lemma \ref{lem:BermanMIS1} applies. Hence $G'$ contains some induced subgraph $G'[X']$ of size $m \leq 4 \log_2 n - 1$ with strictly more edges than vertices. Now select $m+1$ edges in $G'[X']$ and replace the edges that were chains in $G[U]$ by their respective chains of degree 2 vertices. Denote the resulting graph, obtained from $m+1$ edges from $G'[X']$ by expanding the chains, by $G[X]$ (note that $G[X]$ is indeed an induced subgraph of $G$ so this notation is valid).

Note that in the expanding of the chains, at most $m+1$ chains are expanded and every chain has size less than $h$. So $G[X]$ has in total at most $m + (m+1)(h-1) = (m+1)h - 1$ vertices. As $m \leq 4 \log_2 n - 1$ we see that $V_X = (m+1)h - 1 < 4 h \log_2 n$ vertices and the proof is complete.
\end{proof}
The following lemma is an immediate consequence.
\begin{lemma}\label{lem:BermanMIS3}
(Special case of Lemma \ref{lem:BermanMIS2}) Let $\varepsilon > 0$. If in an undirected multigraph $G = (V,E)$ there is no subset of vertices $X$ of size at most $4 (1 + \frac{1}{\varepsilon}) \log_2 n$ such that in $G[X]$ there are more edges than vertices, then $|E| \leq (1 + \varepsilon) |V|$.
\end{lemma}

\subsection{A new bound on the integrality gap}\label{subsec:TheoremLP}

Using this last lemma, the following improved bound on the integrality gap of the intersecting family LP can be established.
\begin{theorem}\label{thm:IntegralityGap2}
Let $\varepsilon > 0$. The ratio between any LP solution to the intersecting family LP and any $4 ( 1 + \frac{1}{\varepsilon} ) \log_2 n$-locally optimal solution is at most $\frac{k}{3} + 1 + \varepsilon$. Thus the integrality gap of the intersecting family LP is at most $\frac{k}{3} + 1 + \varepsilon$. \textbf{[NOTE: see page iii]}
\end{theorem}
\begin{proof}
Let $M$ be a $4 ( 1 + \frac{1}{\varepsilon} ) \log_2 n$-locally optimal matching. Let $x$ be a feasible solution to the intersecting family LP, and let $\mathcal{F}$ be the set of hyperedges with $x_e > 0$. To prove the theorem it suffices to prove that $x(\mathcal{F}) \leq \left( \frac{k}{3} + 1 + \varepsilon \right) |M|$. Denote by $\mathcal{F}_1$, $\mathcal{F}_2$ and $\mathcal{F}_{3+}$ the subsets of $\mathcal{F}$ in which every hyperedge in intersects exactly one, exactly two or at least three hyperedges in $M$, respectively.

Note that $M$ is in particular a 1-local and a 2-locally optimal matching. Since $M$ is a 1-locally optimal matching, each hyperedge $e$ in $\mathcal{F}$ intersects at least one hyperedge in $M$: otherwise $M \cup \{e\}$ would be a larger matching. Hence $\mathcal{F} = \mathcal{F}_1 \cup \mathcal{F}_2 \cup \mathcal{F}_{3+}$. We will now proceed to bound $x(\mathcal{F}_1)$ and $x(\mathcal{F}_2)$ in terms of $|M|$.

Consider a hyperedge $e$ in $M$ and define $\mathcal{F}_1(e) := \left\{ f \in \mathcal{F}_1 \mid f \cap e \neq \emptyset \right\}$. Then $\mathcal{F}_1(e)$ is an intersecting family. For suppose this is false, then there are two disjoint hyperedges $f_1, f_2$ in $\mathcal{F}_1$. These are disjoint from all other hyperedges in $M \setminus \{e\}$ because $f_1, f_2 \in \mathcal{F}_1$. Therefore $M - e + f_1 + f_2$ is a larger matching than $M$, but this is in contradiction with the fact that $M$ is a 2-locally optimal matching. Hence $\mathcal{F}_1(e)$ is an intersecting family. Thus by the intersecting family constraint $x(\mathcal{F}_1(e)) \leq 1$. Now summing over all hyperedges $e \in M$ yields
\begin{equation}\label{F1}
x(\mathcal{F}_1) \leq |M|.
\end{equation}

For the purpose of the analysis of the bound on $x(\mathcal{F}_2)$, consider an auxiliary multigraph $H$ with a vertex for every set in $M$. Two vertices $m_1, m_2 \in M$ are adjacent in $H$ if and only if there is a set in $\mathcal{F}_2$ that intersects both sets corresponding to $m_1$ and $m_2$. Note that $H$ consists of exactly $|M|$ vertices and $|\mathcal{F}_2|$ edges, potentially some of them are parallel. By assumption, $M$ cannot be improved by a set of size at most $4 ( 1 + \frac{1}{\varepsilon} ) \log_2 n$. As an induced subgraph of $H$ with more edges than vertices constitutes an improving set, $H$ cannot contain an induced subgraph of size at most $4 ( 1 + \frac{1}{\varepsilon} ) \log_2 n$ with more edges than vertices. Hence, by Lemma~\ref{lem:BermanMIS3} we infer that $|E(H)| \leq |V(H)| (1 + \varepsilon)$. Consequently, $|\mathcal{F}_2| \leq (1 + \varepsilon) |M|$. As $x(\mathcal{F}_2) \leq |\mathcal{F}_2|$ it follows that
\begin{equation}\label{F2}
x(\mathcal{F}_2) \leq (1 + \varepsilon) |M|.
\end{equation}

Now that $x(\mathcal{F}_1)$ and $x(\mathcal{F}_2)$ are upper bounded in terms of $|M|$, the result follows easily. Note that there are $k|M|$ vertices in $M$. Then the degree constraint yields
\begin{align*}
k|M| & \geq x(\mathcal{F}_1) + 2 x(\mathcal{F}_2) + 3 x(\mathcal{F}_{3+}) \\
     & = 3 x(\mathcal{F}) - x(\mathcal{F}_2) - 2 x(\mathcal{F}_1).
\end{align*}
Now plug in the bound from \eqref{F1} and \eqref{F2} to find
\begin{equation*}
k|M| \geq 3 x(\mathcal{F}) - (1 + \varepsilon) |M| - 2 |M|,
\end{equation*}
which can be rewritten as
\begin{equation*}
x(\mathcal{F}) \leq \frac{k + 3 + \varepsilon}{3} |M| \leq \left( \frac{k}{3} + 1 + \varepsilon \right) |M|. \qedhere
\end{equation*}
\end{proof}

%For completeness, another way to finish the proof is by writing the following.
%%
%\begin{equation*}
%x(\mathcal{F}_{3+}) \leq \frac{k|M| - x(\mathcal{F}_2) - 2 x(\mathcal{F}_1)}{3}.
%\end{equation*}
%%
%Now plug in the bounds for $x(\mathcal{F}_1)$, $x(\mathcal{F}_2)$ and $x(\mathcal{F}_{3+})$ in the equation
%%
%\begin{equation*}
%x(\mathcal{F}) = x(\mathcal{F}_1) + x(\mathcal{F}_2) + x(\mathcal{F}_3),
%\end{equation*}
%%
%and rewrite to find an integrality gap bounded by $\frac{k}{3} + 1 + \varepsilon$.

%Where the results for the standard LP extend to the weighted case, the results for the intersecting family LP do not. This is for the same reason why the unweighted approximation algorithms for $k$-set packing do not easily generalize to the weighted case: the local search technique is relying crucially on cardinality. For example, $\mathcal{F}_1(e)$ as in the previous proof need not be an intersecting family in the weighted case: the objective function might increase when we add less sets than we remove. It is not obvious how to improve significantly on the integrality gap of the standard LP for weighted $k$-set packing.

\section{A polynomially sized LP}\label{sec:SizeLP}

A new bound on the integrality gap of the intersecting family LP has now been established, but the LP might have exponentially many constraints and thus not solvable in polynomial time. This section treats the result from \cite{LapChiLau} that for constant $k$ only a polynomial number of constraints can be added to the standard LP formulation such that every intersecting family has a fractional value of at most 1. To this end the definition of a kernel is needed. Intuitively, for an intersecting family $K$ its kernel is a subset $U$ of the vertices covered by $K$ such that all hyperedges restricted to $U$ still form an intersecting family. More formally, for each hyperedge $e$ define $e_U = e \cap U$, and for a collection of hyperedges $K$ define $K_U = \{ e_u \mid e \in K \}$. Then $U$ is a kernel for an intersecting family $K$ if $K_U$ is an intersecting family.

We proceed with the following result from \cite{OriginalKernel}.
\begin{lemma}\label{lem:OriginalKernel}
(\cite{OriginalKernel}) For every $k$ there exists an $f(k)$ such that for every $k$-uniform intersecting family $K$ there is a kernel $S$ of cardinality at most $f(k)$.
\end{lemma}
The point of this lemma that will be exploited is that the size of this kernel $f(k)$ is independent of $n$ or the number of vertices of the hypergraph. This was a well-studied topic and the interested reader may read \cite{Kernel1,Kernel2,Kernel3,Kernel5,Kernel6,Kernel4}.
Now we will repeat the result from \cite{LapChiLau} that shows how to apply this lemma to prove the following theorem, with our improved bound on the integrality gap.
\begin{theorem}\label{thm:LP2}
Let $\varepsilon > 0$. There is a polynomially sized LP for $k$-set packing with integrality gap at most $\frac{k}{3} + 1 + \varepsilon$. \textbf{[NOTE: see page iii]}
\end{theorem}
\begin{proof}
We will prove it is possible to find a polynomially sized LP that still captures all the constraints of the intersecting family LP. This immediately implies the LP has the claimed integrality gap by Theorem \ref{thm:IntegralityGap2}. We follow the proof from \cite{LapChiLau}.

Let $G = (V,E)$ be the hypergraph and create a variable $x_U$ for every subset $U \subseteq V$ that is a subset of some hyperedge $e \in E$. To enforce that $x_U$ represents the fractional value of $U$, add the constraint $x_U = \sum_{e \supseteq U} x_e$. Each new variable $U$ is a subset of some hyperedges in $G$, and $U$ is said to be contained in a subset $S$ if $U \subseteq S$.

Now enumerate all possible subsets $S \subseteq V$ with $|S| \leq f(k)$. For each such subset $S$, enumerate all possible intersecting families $K_S$ formed by the new variables contained in $S$. For each such intersecting family $K_S$ write the following kernel constraint.
\begin{equation*}
\sum_{U \in K_S} x_U \leq 1.
\end{equation*}
There are $\sum_{i=1}^{f(k)} {n \choose i} \leq n^{f(k) + 1}$ possible kernels. For each kernel $S$ with $|S| = t$ there are at most $2^t$ new variables contained in $S$, because there are at most $2^t$ subsets of $S$. Hence there are at most $2^{2^t}$ intersecting families $K_S$ induced in $S$, because there are at most $2^{2^t}$ hypergraphs in $S$. Every such intersecting family corresponds to one constraint, so there are no more than $n^{f(k)+1} 2^{2^{f(k)}}$ kernel constraints. So when $k$ is a constant this is a number polynomial in $n$.

By Lemma \ref{lem:OriginalKernel} each intersecting family has a kernel constraint and hence each intersecting family has at most a fractional value of 1. So indeed, the intersecting family LP can be rewritten into a polynomially sized LP for every constant $k$.
\end{proof} 
\chapter{SDP formulation}\label{chap:SDP}

The first two sections treat background on semidefinite programming and the Lov\'{a}sz Theta function. These concepts are needed for the proof in Section \ref{sec:SDP} of the theorem that there exists a polynomially sized SDP with the improved integrality gap from Theorem \ref{thm:IntegralityGap2}. %We will treat the result that some SDP for $k$-set packing captures all constraints of the intersecting family LP \cite{LapChiLau}, leading to the theorem.

\section{Background on semidefinite programming}\label{sec:Background}

A semidefinite program is a more general form of a linear program. %and this section is devoted to this generalization.
General references for a thorough background about SDPs are \cite{SDP11,SDP2,SDP10,SDP5,SDP1}.

In a linear program, the objective is to maximise a linear function over a convex polyhedron. In a semidefinite program the numbers are substituted by vectors and the dot product of two vectors is used instead of the multiplication of two numbers. A semidefinite program can be written in the following form. %where we are minimizing over the vectors $x^1, \ldots, x^n \in \mathbb{R}^n$.
%
%\begin{equation*}
%\begin{alignedat}{2}
%\text{max}  \quad & \sum_{i,j} c_{ij} (x^i x^j) \ \\
%\text{s.t.} \quad & \sum_{i,j} a_{ijk} (x^i x^j) \leq b_k, & \quad k = 1, \ldots, m
%\end{alignedat}
%\end{equation*}
%%
%However, it is more convenient to think of semidefinite programs in another equivalent way. We consider semidefinite programs of the following form.
%
\begin{equation*}%\tag{Primal SDP}\label{Primal SDP}
\begin{alignedat}{2}
\text{max}  \quad & c^T x \ \\
\text{s.t.} \quad & x_1 A_1 + \ldots + x_n A_n - B \succeq 0
\end{alignedat}
\end{equation*}
Here $x \in \mathbb{R}^n$ is the vector of decision variables one needs to assign values to in order to maximise the inner product $c^T x$ with the given vector $c \in \mathbb{R}^n$. $A_1, \ldots, A_n, B \in \mathbf{S}^m$ are given symmetric $m \times m$ matrices, so one can think of $X := x_1 A_1 + \ldots + x_n A_n - B$ as a matrix whose entries are linear functions over the variables $x_i$. The constraint $X \succeq 0$ means $X$ needs to be positive semidefinite, which is equivalent to $y^T X y$ being nonnegative for all $y \in \mathbb{R}^m$ or to $X$ having only nonnegative eigenvalues. When $X \succeq 0$ for some $x \in \mathbb{R}^n$ we say $x$ is a feasible solution. Since both the objective function and the constraints are convex in $x$, a semidefinite program is a convex optimization problem. In contrast to a linear program, its feasible region is in general not a a polyhedron.

\section{Background on the Lov\'{a}sz Theta function}\label{sec:Theta}

This section gives some background on the famous Lov\'{a}sz Theta function introduced in \cite{Shannon}. %Historically, the Lov\'{a}sz Theta function $\vartheta(G)$ for a graph $G$ has been introduced to upper bound the Shannon capacity $\Theta(G)$ \cite{Shannon}. Since the concept of the Shannon capacity is not required for the results on $k$-set packing we will not treat it here. For a long time it was an open question if the Shannon capacity of $C_5$ was equal to $\sqrt{5}$ or not, and in his paper Lov\'{a}sz \cite{Shannon} proved this in the affirmative and generalized his method to obtain upper bounds on the Shannon capacity of any graph: the Lov\'{a}sz Theta function. We have $\alpha(G) \leq \Theta(G) \leq \vartheta(G)$ for all graphs $G$.
We introduce the Lov\'{a}sz Theta function via orthogonal representations. In order to do that some background about the stable set polytope is first given in Subsection \ref{subsec:STAB}. Subsection \ref{subsec:ONR} will talk about orthogonal representations and then the Lov\'{a}sz Theta function is introduced in Subsection \ref{subsec:Theta}. General references for this section are \cite{SDP11,GLS,Lovasz1,SDP10}.

\subsection{The stable set polytope}\label{subsec:STAB}

%Before we continue with the results for a semidefinite program for $k$-set packing, we need some background on the famous Lov\'{a}sz Theta function. In order to see the connection of this function with the $k$-set packing problem, we first need some background about the stable set polytope. This is the topic in this section, and in the next section we treat the Lov\'{a}sz Theta function. General references for this section are \cite{SDP11,GLS,Lovasz1,SDP10}.

\paragraph{Definition} Stable set is another name for an independent set. Given a graph $G = (V,E)$, $\alpha(G)$ denotes the size of the maximum independent set in $G$. For every subset of the vertices $S \subseteq V$ its incidence vector is denoted by $\chi^S \in \mathbb{R}^V$, i.e. for all $i \in V$, $\chi^S_i = 1$ if $i \in S$ and $\chi^S_i = 0$ otherwise. Now define the stable set polytope STAB$(G)$ as follows.
\begin{equation*}
\textrm{STAB}(G) = \textrm{conv.hull}( \chi^S \in \mathbb{R}^V \mid S \textrm{ is an independent set in } G )
\end{equation*}
\paragraph{Properties} So STAB$(G)$ is the smallest convex set in $\mathbb{R}^V$ containing the incidence vectors of all independent sets. Since all extreme points of this polytope are $0,1$-vectors, there is a system of linear inequalities describing this convex set. Theoretically it is possible to find $\alpha(G)$ by optimising the linear objective function $\sum_i x_i$ over STAB$(G)$. However, the number of constraints is generally exponential in $|V|$ so this is not an efficient approach to find $\alpha(G)$, which should be expected as determining $\alpha(G)$ is NP-hard. What one can do, however, is to find extra inequalities for the stable set polytope and find upper bounds for $\alpha(G)$.

\paragraph{The clique constrained stable set polytope} With the intersecting family LP in mind it is natural to start with the following inequalities for $x \in \mathbb{R}^V$.
\begin{equation}\label{eq:STAB1}
x_i \geq 0 \quad i \in V,
\end{equation}
%
%\begin{equation}\label{eq:STAB2}
%x_i + x_j \leq 1 \quad \{i,j\} \in E.
%\end{equation}
%%
%Let us write FSTAB$(G)$ for the polytope generated by these constraints.
%%
%\begin{equation*}
%\textrm{FSTAB}(G) = \textrm{conv.hull}( x \in \mathbb{R}^V \mid \textrm{Constraints } \eqref{eq:STAB1} \textrm{ and } \eqref{eq:STAB2} \textrm{ hold} ).
%\end{equation*}
%%
%Now any independent set in $G$ corresponds to an integral vertex in FSTAB$(G)$ and vice versa. For bipartite graphs without isolated vertices FSTAB$(G)$ and STAB$(G)$ coincide. If we add the constraint $x_i \leq 1$ for all $i \in V$ they coincide for any bipartite graph. However, for general graphs FSTAB$(G)$ may have other non-integral vertices and FSTAB$(G)$ is larger than STAB$(G)$. Consider $G = K_3$ for example, a triangle. Then the maximum independent set in $G$ has size one, but FSTAB$(G)$ also contains the vertex $(\frac{1}{2},\frac{1}{2},\frac{1}{2})$. In a general $K_n$, FSTAB$(G)$ contains the vertex $x = (\frac{1}{n-1}, \ldots, \frac{1}{n-1})$ with $\sum_i x_i = \frac{n}{n-1} > 1$.
%
%This should remind the reader of the standard LP relaxation and the intersecting family LP described in Chapter \ref{chap:LP}. The intersecting family LP had an extra constraint for every intersecting family of hyperedges, which form a clique in the conflict graph. From every clique only one vertex can be contained in an independent set, so add the following constraint.
%
\begin{equation}\label{eq:STAB3}
\sum_{i \in V(Q)} x_i \leq 1 \quad \textrm{for all cliques } Q \textrm{ in } G.
\end{equation}
%
%The first polytope approximating $STAB(G)$ is the clique constrained stable set polytope, defined as follows.
Now define the clique constrained stable set polytope as follows.
\begin{equation*}
\textrm{QSTAB}(G) = \textrm{conv.hull}( x \in \mathbb{R}^V \mid \textrm{Constraints } \eqref{eq:STAB1} \textrm{ and } \eqref{eq:STAB3} \textrm{ hold} ).
\end{equation*}
Now any independent set in $G$ corresponds to an integral vertex in QSTAB$(G)$ and vice versa. The clique constrained stable set polytope is the first approximation of the stable set polytope that is considered here, but to formally define the Lov\'{a}sz Theta function another polytope is introduced in the next subsection.

\subsection{Orthogonal representations}\label{subsec:ONR}

\paragraph{Definition} This subsection introduces the Lov\'{a}sz Theta function via orthogonal representations. Let $G = (V,E)$ be a graph and let $\overline{E} = \{ \{i,j\} \in V \times V \mid \{i,j\} \not\in E \}$ be the complement of $E$. Formally, an orthogonal representation of $G$ is a mapping (labeling) $u: V \rightarrow \mathbb{R}^d$ for some $d$ such that $u_i^T u_j = 0$ for all $\{i,j\} \in \overline{E}$. In other words we need to assign a vector $u_v$ to every vertex $v$ such that the vectors of any two non-adjacent vertices are perpendicular to each other. Such a mapping always exists, in fact, there are two trivial mappings: map all vertices to 0, or map the vectors to a set of mutually orthogonal vectors in $\mathbb{R}^V$.

\paragraph{Orthogonal constrained stable set polytope}
An orthogonal representation $(u_i \mid i \in V)$ with $u_i \in \mathbb{R}^d$ is called orthonormal when every vector has unit length, i.e. when $\|u_i\| = 1$ for all $i \in V$. Let $c$ be some vector in $\mathbb{R}^d$ with $\|c\| = 1$ (for example, take $c = e_1$). Then for any stable set $S \subseteq V$ its vectors $\{u_i \mid i \in S\}$ are mutually orthonormal as the vertices are non-adjacent, and hence
\begin{equation*}
\sum_{i \in S} (c^T u_i)^2 \leq 1.
\end{equation*}
This is true because the left-hand side is the squared length projection of $c$ onto the subspace spanned by the $u_i$. The length of this projection is at most the length of $c$ which is 1. In fact, note that $\sum_{i \in V} (c^T u_i)^2 \chi^S = \sum_{i \in S} (c^T u_i)^2$, which yields that the following inequality holds for the incidence vector $\chi^S$ of any stable set $S \subseteq V$. It is called the orthogonality constraint.
\begin{equation}\label{eq:STAB5}
\sum_{i \in V} (c^T u_i)^2 x_i \leq 1.
\end{equation}
Similar like before, we can now define the orthogonal constrained stable set polytope as follows.
\begin{equation*}
\textrm{TSTAB}(G) = \textrm{conv.hull}( x \in \mathbb{R}^V \mid \textrm{Constraints } \eqref{eq:STAB1} \textrm{ and } \eqref{eq:STAB5} \textrm{ hold} ).
\end{equation*}

\subsection{The Lov\'{a}sz Theta function}\label{subsec:Theta}

%In general TSTAB$(G)$ is not a polytope. It is a convex set however, as it is the intersection of infinitely many halfplanes.
The most interesting property of TSTAB$(G)$ is the fact that one can optimise linear functions over it in polynomial time \cite[Theorem 9.3.30]{GLS}. %The relation of TSTAB$(G)$ to the clique constrained stable set polytope is considered in Section \ref{sec:SDP} as it is needed for the result on an SDP relaxation for $k$-set packing treated there.
%\subsection{Basic facts about the Lov\'{a}sz Theta function $\vartheta(G)$}\label{subsec:Theta1}
We can now succinctly define the Lov\'{a}sz Theta function.
\begin{equation*}
\vartheta(G) = \max \left\{ \sum_i x_i \mid x \in \textrm{TSTAB}(G) \right\}.
\end{equation*}
%
%In literature the leaning of an orthonormal representation is defined as $\sum_{i \in V} (c^T u_i)^2$, then $\vartheta(G)$ is the maximum leaning of any orthogonal representation.
The following is equivalent by writing out the definitions. Let ONR denote an orthonormal representation. Denote the following LP by \eqref{ThetaLP}.
\begin{equation}\tag{$\vartheta$-LP}\label{ThetaLP}
\begin{alignedat}{2}
\vartheta(G) = \text{max}  \quad & \sum_i x_i \ \\
               \text{s.t.} \quad & \sum_{i \in V} (c^T u_i)^2 x_i \leq 1 & \quad \forall c : \|c\| & = 1 \quad \forall \textrm{ONR}\{u_i\} \\
                                 & x_i \geq 0                            & \quad \forall i \in V
\end{alignedat}
\end{equation}

There are a lot of alternative and equivalent definitions for the Lov\'{a}sz Theta function. The semidefinite program for $k$-set packing uses the following equivalent definition, known in the literature as $\theta_3(G)$. An orthogonal representation $\{b_i\}$ is called normalised if $\sum_i \|b_i\|^2 = 1$, and define $\overline{G} = (V,\overline{E})$.
%
%\subsection{Alternative definitions of $\vartheta(G)$}\label{subsec:Theta2}
%
%Take the setting from the previous subsection and call an orthogonal representation $\{b_i\}$ normalized if $\sum_v \|b_v\|^2 = 1$. Then define the following quantity.
%
%\begin{equation*}
%\vartheta_1(G) = \min_{c,\{u_i\}} \max_{i \in V} \frac{1}{(c^T u_i)^2}.
%\end{equation*}
%%
%In order to introduce the other formulations we define the following. We call a matrix $A$ feasible for $G$ if it is a real symmetric matrix indexed by the vertices of $G$ with $A_{uu} = 1$ for all $u$ and $A_{uv} = 1$ if $uv \not\in E$. The other elements can be any real value (as long as $A$ is symmetric). For such a matrix $A$ let $\Lambda(A)$ be its maximum eigenvalue, or equivalently, $\Lambda(A) = \max\{x^T A x \mid \|x\| = 1\}$. Now define
%%
%\begin{equation*}
%\vartheta_2(G) = \min \{ \Lambda(A) \mid A \textrm{ is feasible for } G \}.
%\end{equation*}
%%
%Now call an orthogonal representation $\{b_i\}$ normalized if $\sum_v \|b_v\|^2 = 1$. Then
%%
\begin{equation*}
\vartheta_3(G) = \max \left\{ \sum_{u,v} b_u b_v \mid b \textrm{ is a normalised orthogonal representation of } \overline{G} \right\}.
\end{equation*}
%%
%Finally, let $c \in \mathbb{R}^d$ range over all vectors of unit length and let $\{b_i\}$ range over all orthonormal representation of $\overline{G}$ to define a quantity bearing some similarities to $\vartheta_1(G)$:
%%
%\begin{equation*}
%\vartheta_4(G) = \max_{d,\{b_i\}} \sum_{i \in V} (c^T b_i)^2.
%\end{equation*}
%%
%Then we have
%%
%\begin{lemma}(\cite[Theorem 9.3.12]{GLS},\cite[Theorem 12.1]{Lovasz1})\label{lem:Theta}
%For any graph $G$ we have
%%
%\begin{equation*}
%\vartheta(G) = \vartheta_1(G) = \vartheta_2(G) = \vartheta_3(G) = \vartheta_4(G).
%\end{equation*}
%\end{lemma}

%This is all the background we needed on semidefinite programming and the Lov\'{a}sz Theta function to state the theorem in the next section.

\section{An SDP for $k$-set packing}\label{sec:SDP}

%Given this background
This section contains the main theorem of this chapter. %Now apply this result to $k$-set packing as follows.
As before, view $k$-set packing as the independent set problem in a $k+1$-claw free graph and consider the following clique LP.
\begin{equation}\tag{Clique LP}\label{eq:CliqueLP}
\begin{alignedat}{2}
\text{max}  \quad & \sum_i x_i \ \\
\text{s.t.} \quad & x \in \textrm{QSTAB}(G)
\end{alignedat}
\end{equation}
This is optimising the size of an independent set over the clique constrained stable set polytope.
\begin{lemma}\label{lem:SDPnew1}
\eqref{eq:CliqueLP} is equivalent to the intersecting family LP of Chapter \ref{chap:LP}.
\end{lemma}
\begin{proof}
The nonnegativity constraints $x_i \geq 0$ for \eqref{eq:CliqueLP} obviously match the same constraints in the intersecting family LP. The clique constraints $\sum_{i \in Q} x_i \leq 1$ for cliques of size 1 imply the bound $x_i \leq 1$. Evidently they also imply the intersecting family constraints $x(K) \leq 1$ for $Q = K$. The constraints that $x(\delta(v)) \leq 1$ are also implied by the clique constraints: all hyperedges covering element $v$ form a clique in the conflict graph. The other way around is similar.
\end{proof}
Note that replacing QSTAB$(G)$ by TSTAB$(G)$ yields $\vartheta(G)$. These are related as follows.
\begin{lemma}(\cite[Lemma 4.3]{LapChiLau})\label{lem:SDPnew2}
Any feasible solution to \eqref{ThetaLP} is a feasible solution to \eqref{eq:CliqueLP}.
\end{lemma}
\begin{proof}
It suffices to show that the orthogonality constraints \eqref{eq:STAB5} imply the clique constraints \eqref{eq:STAB3}. Let $Q$ be a clique in $G$ and map all vertices of $Q$ to $c$ and all other vertices to mutually orthogonal vectors that are also orthogonal to $c$. Then the orthogonality constraint for $Q$ implies its clique constraint.
\end{proof}
In other words, for every graph $G$
\begin{equation*}%\label{eq:inclusions}
\textrm{STAB}(G) \subseteq \textrm{TSTAB}(G) \subseteq \textrm{QSTAB}(G).
\end{equation*}
Finally all linear and semidefinite programs can be linked.
\begin{lemma}\label{lem:SDPnew3}
$\vartheta_3(G)$ is a stronger relaxation than the intersecting family LP.
\end{lemma}
\begin{proof}
By Lemma \ref{lem:SDPnew2} \eqref{ThetaLP} is a stronger relaxation than the clique LP. Hence by Lemma \ref{lem:SDPnew1} \eqref{ThetaLP} is also stronger than the intersecting family LP. \eqref{ThetaLP} is equivalent to $\vartheta(G)$, which is equivalent to $\vartheta_3(G)$. Hence $\vartheta_3(G)$ is a stronger relaxation than the intersecting family LP.
\end{proof}
$\vartheta_3(G)$ can be written as follows.
\begin{equation}\tag{$\vartheta_3$-LP}
\label{Theta3LP}
\begin{alignedat}{2}
\vartheta_3(G) = \text{max}  \quad & \sum_{i,j \in V} u_i u_j \ \\
                 \text{s.t.} \quad & u_i u_j = 0,             & \quad \forall (i,j) \in E \\
                                   & \sum_{i=1}^n u_i^2 = 1   \\
                                   & u_i \in \mathbb{R}^d,    & \forall i \in V \quad \verb" "
\end{alignedat}
\end{equation}
This is a semidefinite program which is a stronger relaxation than the intersecting family LP. Then by Lemma \ref{lem:SDPnew3} and Theorem \ref{thm:IntegralityGap2} the following is true.
\begin{theorem}\label{thm:SDP2}
Let $\varepsilon > 0$. \eqref{Theta3LP} is an SDP relaxation for $k$-set packing with integrality gap at most $\frac{k}{3} + 1 + \varepsilon$. \textbf{[NOTE: see page iii]}
\end{theorem}
This is Theorem 1.5 from \cite{LapChiLau} with our improved bound on the integrality gap. Since this semidefinite program has a polynomial size, the following theorem is also true, similar to Theorem \ref{thm:LP2}.
\begin{theorem}\label{thm:SDP3}
Let $\varepsilon > 0$. There is a polynomially sized SDP for $k$-set packing with integrality gap at most $\frac{k}{3} + 1 + \varepsilon$. \textbf{[NOTE: see page iii]}
\end{theorem}

%\section{Extending the results to the weighted case}\label{sec:WeightedSDP} 
\chapter{Weighted approximation}\label{chap:Weighted}

In this chapter we give a simplified proof of the main lemma of Berman's paper \cite{Berman} containing the currently best approximation algorithm for weighted $k$-set packing with approximation guarantee $\frac{k+1}{2}$. First the necessary terminology is introduced in Section \ref{sec:Terminology}. Then the algorithm is discussed in Section \ref{sec:Algorithm} and Section \ref{sec:ProofBerman} provides a new proof of the main lemma.

\section{Terminology}\label{sec:Terminology}

\paragraph{Claws} Consider the following setting from \cite{Berman} for this chapter. For a graph, define a $k$-claw $C$ as a subgraph isomorphic to $K_{1,k}$, the complete bipartite graph on 1 and $k$ vertices. For convenience, define a 1-claw to be a singleton set $C$ with $T_C = C$. A claw is a $k$-claw for some $k$. Define the single vertex connected to all other vertices of the claw to be the center $Z_C$. The other vertices of the claw (forming an independent set by definition) are called the talons $T_C$ of the claw. A $k$-claw has $k$ talons and one center vertex. Write $C = Z_C \cup T_C$ for a claw $C$ with center vertex $Z_C$ and talons $T_C$.

\paragraph{Approximation guarantee} Let $G = (V,E)$ be a $k$-claw free graph with a weight $w(v)$ for every vertex $v \in V$. The main theorem of \cite{Berman} is a $\frac{k}{2}$-approximation algorithm for maximum independent set in such a graph. This yields a $\frac{k+1}{2}$-approximation algorithm for weighted $k$-set packing because any packing corresponds to an independent set in a $k+1$-claw free graph. The algorithm searches for claws satisfying certain properties and then adds the talons of the claw to the current independent set $A$ and removes the neighbours of the talons in $A$.

\paragraph{Notation} We will use the following notational conventions. Define for a vertex $v \in V$ its open neighbourhood (or just its neighbourhood) $N(v) = \{ w \in V \mid \{v,w\} \in E \}$ and its closed neighbourhood $N[v] = N(v) \cup \{v\}$. For a subset of the vertices $W \subseteq V$ define its closed neighbourhood $N[W] = \bigcup_{w \in W} N[w]$ and write $N(W) = N[W] \setminus W$ for the (open) neighbourhood of $W$.
\begin{defn}
For two given subsets of the vertices $U$ and $A$ write $N(U,A) = N(U) \cap A$, i.e. the neighbourhood of $U$ in $A$.
\end{defn}
$N(U,A)$ is called the $A$-neighbourhood of $U$ and we refer to the vertices in $N(U,A)$ as the $A$-neighbours of $U$. Write $N(u,A)$ for $N(\{u\},A)$ for some vertex $u$ and some subset of the vertices $A$.
\begin{defn}
By $n(u,A)$, denote a vertex $v \in N(u,A)$ with maximum weight.
\end{defn}
If this is not uniquely defined, simply choose a random vertex of the different possibilities. %For a given independent set $A$ we denote $n(u) = n(u,A)$ for short, being a vertex in $N(u,A)$ with maximum weight.

For convenience we introduce the notation $w(U) = \sum_{v \in U} w(v)$ for some subset of the vertices $U \subseteq V$. Even shorter, we write the following.
\begin{defn}
$w(U,A) = w(N(U,A))$ and $w(u,A) = w(N(u,A))$.
\end{defn}
These are the sums of the weights of the $A$-neighbours of a subset of the vertices $U$ respectively one vertex $u$. These notations extend to different weight functions such as $w^2$, in particular note that $w^2(U,A) = \sum_{v \in N(U,A)} w^2(v) \neq \left( w(U,A) \right)^2$. Now define the following function as in \cite{Berman}.
\begin{equation*}
charge(u,v) = \left\{
                \begin{array}{ll}
                  w(u) - \frac{1}{2}w(u,A), & \hbox{if $v = n(u,A)$;} \\
                  0, & \hbox{otherwise.}
                \end{array}
              \right.
\end{equation*}

Now define the following.
\begin{defn}
Let $A$ be an independent set in a graph $G = (V,E)$. Define a good claw $C = Z_C \cup T_C$ to be a claw satisfying either of the following properties.
\begin{enumerate}
  \item $N(T_C,A) = \emptyset$, i.e. adding $T_C$ to $A$ to obtain another independent set does not require the removal of any sets in $A$; or
  \item The center vertex $Z_c$ is in $A$ and $\sum_{u \in T_C} charge(u,v) > \frac{1}{2}w(v)$.
\end{enumerate}
A claw $C$ is called a nice claw if it is a minimal set forming a good claw, i.e. if there is no strict subset of $C$ forming a smaller good claw.
\end{defn}

\section{Two algorithms joining forces}\label{sec:Algorithm}

\subsection{The algorithms \textsc{SquareImp} and \textsc{WishfulThinking}}\label{subsec:Algorithms}

In this setting, define the following algorithm:
\begin{quote}
\textsc{SquareImp} \\
$A \leftarrow \emptyset$ \\
While there exists a claw $C$ such that $T_C$ improves $w^2(A)$ \\
\verb"  " $A \leftarrow A \cup T_C \setminus N(T_C,A)$
\end{quote}
Now define the following algorithm:
\begin{quote}
\textsc{WishfulThinking} \\
$A \leftarrow \emptyset$ \\
While there exists a nice claw $C$ \\
\verb"  " $A \leftarrow A \cup T_C \setminus N(T_C,A)$
\end{quote}
These two algorithms can now be linked in the following way.
\begin{enumerate}
  \item Every nice claw improves $w^2(A)$, so a run of \textsc{WishfulThinking} forms the initial part of a run of \textsc{SquareImp}. See Section \ref{sec:ProofBerman}.
  \item Consequently, \textsc{WishfulThinking} cannot make more iterations than \textsc{SquareImp}.
  \item When \textsc{SquareImp} terminates it yields an independent set $A$ for which no claw improves $w^2(A)$. Hence there is no more nice claw, so \textsc{WishfulThinking} terminates.
  \item If \textsc{WishfulThinking} terminates, its approximation guarantee is $\frac{k}{2}$. See Subsection \ref{subsec:Approx}.
\end{enumerate}
The proof of the approximation guarantee is repeated in Subsection \ref{subsec:Approx} to get some insight into the non-intuitive definitions of the charge function and good claws. The main lemma is the fact that every nice claw improves $w^2(A)$, for which we give a simplified proof in Section \ref{sec:ProofBerman}.

%Berman \cite{Berman} shows that under the assumption that \textsc{WishfulThinking} terminates, its approximation guarantee is $\frac{k}{2}$, which we will repeat in the next subsection to get some insight into the non-intuitive definitions of the charge function and good claws. He argues why every nice claw improves $w^2(A)$, which is the main lemma for which we give a simplified proof in Section \ref{sec:ProofBerman}. Consequently, \textsc{WishfulThinking} cannot make more iterations than \textsc{SquareImp}. When \textsc{SquareImp} terminates with some independent set $A$, there is no more claw that improves $w^2(A)$, and hence there is no more nice claw. So \textsc{WishfulThinking} terminates and hence the approximation guarantee has been proved.

\subsection{The approximation guarantee}\label{subsec:Approx}

We repeat the following proof of \cite{Berman}.
\begin{lemma}(\cite[Lemma 1]{Berman})\label{lem:Berman1}
Assume that \textsc{WishfulThinking} has terminated and that $A^*$ is an independent set. Then $\displaystyle \frac{w(A^*)}{w(A)} \leq \frac{k}{2}$.
\end{lemma}
\begin{proof}
Let $G = (V,E)$ be the graph and let $A$ be the independent set that has been found using \textsc{WishfulThinking}. Let $A^*$ be any independent set in $G$ (in particular it could be the maximum independent set). We will distribute $w(A^*)$ among the vertices of $A$ such that no vertex $v \in A$ receives more than $\frac{k}{2}w(v)$. This immediately implies the claimed result. The distribution consists of two phases.

In the first phase, every vertex $u \in A^*$ sends to each of its $A$-neighbours $v \in N(u,A)$ a portion of its weight equal to $\frac{1}{2}w(v)$. Note that $N(u,A) \neq \emptyset$ because otherwise $\{u\}$ would be a nice 1-claw and these do not exist when \textsc{WishfulThinking} has terminated.

In this first phase, every vertex $u$ sends a portion of its weight equal to $\frac{1}{2}w(u,A)$. By the definition of the $charge$ function, the portion of its weight that is not distributed yet equals $charge(u,n(u,A))$. In the second phase $u$ sends $charge(u,n(u,A))$ to $n(u,A)$.

Now consider some vertex $v \in A$ in the receiving side of this distribution. In the first phase $v$ gets $\frac{1}{2}w(v)$ from all its neighbours in $A^*$. Because $A^*$ is an independent set and the graph is $k$-claw free, $v$ has at most $k-1$ neighbours in $A^*$. Thus $v$ gets at most $(k-1)\frac{1}{2}w(v) = \frac{k}{2}w(v) - \frac{1}{2}w(v)$ in the first phase.

In the second phase, $v$ receives exactly $\sum_{u \in N(v,A^*)} charge(u,v)$. By the definition of a good claw, this can be at most $\frac{1}{2}w(v)$: otherwise the vertices in $A^*$ sending positive $charge$s to $v$ form the talons of a good claw with $v \in A$ at its center.

Hence every vertex $v \in A$ receives at most $\frac{k}{2}w(v) - \frac{1}{2}w(v) + \frac{1}{2}w(v) = \frac{k}{2}w(v)$, and hence the weight of $A^*$ is at most $\frac{k}{2}$ times as much as the weight of $A$.
\end{proof}

As noted, this yields a $\frac{k}{2}$-approximation for the weighted independent set problem in $k$-claw free graphs, which results in a $\frac{k+1}{2}$-approximation for weighted $k$-set packing.

\paragraph{Instructive example} Berman \cite{Berman} proceeds with an example of an instance where an iteration of \textsc{WishfulThinking} in fact decreases $w(A)$, which is the function it is in fact trying to maximise. This is contra-intuitive: we are using a local search technique and from a given solution we might move to a next solution with a worse objective value. However, as we will prove in Section \ref{sec:ProofBerman}, an iteration of \textsc{WishfulThinking} always increases $w^2(A)$ and that suffices for the analysis.

Here is the example, see Figure \ref{fig:Example}. Let $S$ be a subset of the current independent set $A$ depicted at the bottom of Figure \ref{fig:Example} and let $T$ be a subset of the vertices in $V \setminus A$ depicted at the top. Write $S = \{s_1, \ldots, s_5\}$ with $w(s_i) = 10$ for $i=1,\ldots,5$, and $T = \{t_1, t_2\}$ with $w(t_1) = w(t_2) = 18$. Let $n(t_i,A) = s_3$ for $i=1,2$.

%\begin{figure}
%        \centering
%        \includegraphics[width = \textwidth]{Images/Example.jpg}
%        \caption{An example where an iteration of \textsc{WishfulThinking} actually decreases $w(A)$. [TODO: Make better picture]}
%        \label{fig:Example}
%\end{figure}

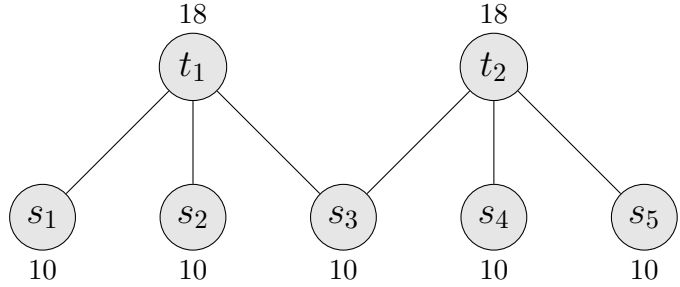
\begin{figure}
\centering
\begin{tikzpicture}[node distance = 2cm, main node/.style={circle,fill=black!10,draw,font=\sffamily\Large\bfseries}]
%[->,>=stealth',shorten >=1pt,auto,node distance=3cm,
  %thick,main node/.style={circle,fill=blue!20,draw,font=\sffamily\Large\bfseries}]
\node[main node, label=below:10] (A) {$s_1$};
\node[main node, label=below:10, right of=A] (B) {$s_2$};
\node[main node, label=below:10, right of=B] (C) {$s_3$};
\node[main node, label=below:10, right of=C] (D) {$s_4$};
\node[main node, label=below:10, right of=D] (E) {$s_5$};
\node[main node, label=above:18, above of=B] (F) {$t_1$};
\node[main node, label=above:18, above of=D] (G) {$t_2$};
\draw (A) -- (F);
\draw (B) -- (F);
\draw (C) -- (F);
\draw (C) -- (G);
\draw (D) -- (G);
\draw (E) -- (G);
\end{tikzpicture}
\caption{An example where an iteration of \textsc{WishfulThinking} actually decreases $w(A)$. The bottom vertices are in $S \subseteq A$ and the top vertices are in $T \subseteq V \setminus A$.}
\label{fig:Example}
\end{figure}

We claim that $\{s_3\} \cup \{t_1, t_2\}$ is a nice claw. To see this, note that $charge(t_i,s_3) = w(t_i) - \frac{1}{2}w(t_i,A) = 18 - \frac{1}{2}(10 + 10 + 10) = 3$. So for $s_3$ we have $\sum_{t_i} charge(t_i,s_3) = 3 + 3 = 6$, which is larger than $\frac{1}{2}w(s_3) = 5$. So all conditions are satisfied and $\{s_3\} \cup \{t_1, t_2\}$ is a good claw, and because it is minimal it is also a nice claw.

However, adding $T$ and removing $N(T,S) = S$ means adding two sets of weight 18 and removing 5 sets of weight 10. Hence $w(A)$ decreases by 14. But the squared weight function increases: the gain is twice $18^2$ and the loss is five times $10^2$, so it increases by $648 - 500 = 148$. In fact, elementary calculus shows that $w^c$ increases in this example for $c > \frac{\log 5 - \log 2}{\log 18 - \log 10} \approx 1.56$, or in a more general setting, for $c > \frac{\log |S| - \log |T|}{\log w(s_i) - \log w(t_i)}$. See also \cite{BermanWeighted2} for results on using the misdirected weight function $w^c$ for some $c \neq 1$.

%\paragraph{Other weight functions than $w$} Intuitively, what is the crucial point why $w^2$ might behave differently than $w$? In general, looking at the squared weight function (or at $w^c$ for any $c>1$) is slightly more biased towards larger weights. When a vertex has weight $m$ and we add 1 to its weight, $w$ increases by 1 while $w^2$ increases by $2m+1$. The square of the weight of one vertex might be more than the weight of two vertices. $w^2$ prefers one vertex of weight 3 to two vertices of weight 2, while $w$ prefers the two vertices of weight 2 to the singly vertex of weight 3. So by guiding the search by $w^2$ rather than $w$, an iteration might decrease the real objective function but it is more difficult to get stuck in an inferior locally optimal solution.

\subsection{A new observation}\label{subsec:Observation}

For the independent set problem in $k$-claw free graphs the use of $w^2$ rather than $w$ can be advantageous due to the following observation. Let $A$ be some independent set and let $u$ be some vertex not in $A$.
\begin{equation*}
\sum_{t \in N(u,A)} w^2(t) \leq \sum_{t \in N(u,A)} \left( w(t) \left( \max_{t \in N(u,A)} w(t) \right) \right) = \left( \max_{t \in N(u,A)} w(t) \right) \sum_{t \in N(u,A)} w(t).
\end{equation*}
This proofs the following observation.
\begin{obs}\label{obs}
$w^2(u,A) \leq n(u,A) w(u,A)$.
\end{obs}

So the squared weight function of a set of vertices is capable of capturing information not only about the sum of the weights but also about the maximum weight. Using this simple observation Berman's proof can be simplified in the next section.

\section{Simplified proof}\label{sec:ProofBerman}

Here is a simplified proof of the main lemma from \cite{Berman} proving that adding a nice claw to the current independent set $A$ improves $w^2(A)$. We believe this proof gives some more insight into what is really happening behind the math thanks to Observation \ref{obs}.
\begin{lemma}\label{lem:Berman2}
If $C$ is a nice claw, then $T_C$ improves $w^2(A)$.
\end{lemma}
\begin{proof}
Let $A$ be the current independent set and let $C = Z_C \cup T_C = \{v\} \cup T$ be a nice claw. We need to show that the weight of what is added ($T$) is more than the weight of what is lost ($N(T,A)$), so we need to show that
\begin{equation}\label{first}
w^2(T) > w^2(T,A),
\end{equation}
or equivalently,
\begin{equation}\label{second}
w^2(T) - w^2(T,A-\{v\}) > w^2(v).
\end{equation}
To proof that for a nice claw \eqref{second} holds, we proof the following claim.
\begin{claim}\label{claim1}
Let $C = \{v\} \cup T$ be a nice claw. Then
\begin{equation*}
w^2(T,A-\{v\}) \leq \sum_{u \in T} \left( w^2(u,A) - w^2(v) \right).
\end{equation*}
\end{claim}
\begin{proof}
%By definition, $w^2(N(T,A))$ is the sum of the squared weights of all neighbours of the talons $T$ in $A$. When we sum over all vertices $u \in T$ and look at $\sum_{u \in T} w^2(N(u,A))$, we count every vertex that is adjacent to $p$ vertices in $T$ exactly $p$ times. In particular, we count $w^2(v)$ exactly $|T|$ times, while on the left-hand side we count it only once. If we thus do not add $w^2(v)$ for every $u \in T$ (by subtracting it again) and add it once at the end, we count $w^2(v)$ exactly once like on the left-hand side. As we may count the squared weights of other vertices multiple times, the right-hand side might sum more positive terms and thus be larger.
By definition, $w^2(T,A-\{v\})$ %equals $\displaystyle \sum_{u \in N(T,A-\{v\})} w^2(u)$.
is the sum of the squared weights of all neighbours of the talons $T$ in $A$ excluding $v$.
Summing over $T$ rather than the neighbourhood of $T$, this can be bounded by $\displaystyle \sum_{u \in T} w^2(u,A-\{v\})$. This is an upper bound, because in this expression vertices that are neighbours of more than one vertex in $T$ are counted multiple times. Therefore, $\displaystyle w^2(T,A-\{v\}) \leq \sum_{u \in T} w^2(u,A-\{v\})$. Noting that $w^2(u,A-\{v\})$ equals $w^2(u,A) - w^2(v)$, the claim follows.
\end{proof}
For the first term in \eqref{second}, write $w^2(T) = \sum_{u \in T} w^2(u)$. Now by Claim \ref{claim1}, the following implies \eqref{second}.
\begin{equation}\label{fourth}
\sum_{u \in T} \left( w^2(u) - w^2(u,A) + w^2(v) \right) > w^2(v).
\end{equation}
To show that \eqref{fourth} holds when $C$ is a nice claw, we proceed to the second claim.
\begin{claim}\label{claim2}
Let $C = \{v\} \cup T$ be a nice claw. Then $v = n(u,A)$ for all $u$ in $T$ and
\begin{equation}\label{fifth}
w(v) < 2 \sum_{u \in T} charge(u,v).
\end{equation}
\end{claim}
\begin{proof}
Equation \eqref{fifth} just follows from the definition stating $\sum_{u \in T} charge(u,v) > \frac{1}{2}w(v)$. Also, as $C$ is a nice claw, it is minimal, implying that every term on the right-hand side of \eqref{fifth} is positive. By the definition of $charge$, this is true only if $v$ is the maximum weight neighbour of $u$ within $A$.
\end{proof}

So proving that \eqref{fourth} holds when $C$ is a nice claw has now been reduced by Claim \ref{claim2} to showing that whenever $charge(u,v)>0$ we have
\begin{equation*}
w^2(u) - w^2(u,A) + w^2(v) \geq 2w(v)charge(u,v).
\end{equation*}
Here we plugged in \eqref{fifth} only once in the right-hand side. By the definition of $charge$, this boils down to proving that
\begin{equation}\label{seventh}
w^2(u) - w^2(u,A) + w^2(v) \geq 2w(u)w(v) - w(v)w(u,A)
\end{equation}
holds whenever $v$ is the maximum weight neighbour of $u$.

From this point on we will deviate from the proof of Berman \cite{Berman}. He now scales the quantities, makes a case distinction and rewrites the equations algebraically until it is clear they are indeed true. However, \eqref{seventh} can be shown more easily using Observation \ref{obs}. Plugging this into \eqref{seventh} yields
\begin{equation*}
w^2(u) - w(v)w(u,A) + w^2(v) \geq 2w(u)w(v) - w(v)w(u,A).
\end{equation*}
The terms $w(v)w(u,A)$ now cancel and what is left is
\begin{equation}\label{eighth}
w^2(u) + w^2(v) \geq 2w(u)w(v),
\end{equation}
which is obviously true as this is equivalent to
\begin{equation*}
(w(u)-w(v))^2 \geq 0.
\end{equation*}
We have now proved that when $C$ is a nice claw, \eqref{first} holds, and thus $T_C$ improves $w^2(A)$.
\end{proof}

%\section*{Improving the approximation guarantee}
%
%The above analysis also shows that we cannot simply improve the approximation guarantee to $\frac{d}{c}$ for some $c > 2$. Let us replace the $\frac{1}{2}$ in the definition of $charge(u,v)$ to $\frac{1}{c}$ and the $\frac{1}{2}$ in the definition of a nice (good) claw to $\frac{1}{c}$.
%
%Lemma 1 in Berman can now be easily adapted to hold for any $c>2$. The proof remains exactly the same and \textsc{WishfulThinking} now has an approximation guarantee of $\frac{d}{c}$.
%
%However, Lemma 2 (as above) fails now. The analysis remains exactly the same except for the fact that in \eqref{eighth}, the factor 2 in the right-hand side changes into a $c$. However, with the current analysis, \eqref{eighth} holds only nontrivially for $c=2$. Hence, $w^2$ improves only when $c=2$.
%
%Trying to improve $w^p$ for some $2 < p \in \mathbb{N}$ (also using $c$ rather than 2) yields the same analysis, where \eqref{eighth} translates to
%%
%\begin{equation*}
%w^p(u) + w^p(v) \geq c w(u) w^{p-1}(v).
%\end{equation*}
%%
%This is trivially true for $p = c = 2$. Using the binomial theorem on $(w(u)+w(v))^p$ cancels the term on the right-hand side by letting $c={p \choose 1}=p$, but yields some other terms that seem to hurt.
%
%Note that this is also true for $p = c = 1$, which gives the standard $d$-approximation back. 
\chapter{Discussion}\label{chap:Discussion}

In this chapter we will discuss possible directions for further research. The LP and the SDP relaxations for the problem are discussed in Section \ref{sec:DiscLP}. Section \ref{sec:DiscBerman} considers possible ways in which the weighted approximation algorithm of Berman \cite{Berman} could be lightly changed. Section \ref{sec:IS} continues with a discussion about the difference in the results on the independent set problem in bounded degree graphs and the current results on $k$-set packing. Finally Section \ref{sec:WeightedVSUnweighted} discusses the problems arising when one tries to generalise the unweighted approximation algorithms with weights. Sections \ref{sec:DiscLP}, \ref{sec:IS} and \ref{sec:WeightedVSUnweighted} suggest future research directions and contain some conjectures. %and what makes the $\frac{k+2}{3}$-approximations so considerably easier than the $\frac{k+1}{3}$-approximations.

%In this chapter we will discuss possible directions for further research. We start with a discussion about the LP and the SDP relaxations for the problem. Next we consider possible ways in which the weighted approximation algorithm of Berman \cite{Berman} could be lightly changed. Then we discuss the difference in the results on the independent set problem in bounded degree graphs and the current results on $k$-set packing. We discuss the problems arising when one tries to generalize the unweighted approximation algorithms with weights and what makes the $\frac{k+2}{3}$-approximations so considerably easier than the $\frac{k+1}{3}$-approximations.

\section{LP and SDP relaxations}\label{sec:DiscLP}

Here is a summary of the current results.
\begin{enumerate}
  \item The standard LP relaxation for $k$-set packing has integrality gap $k - 1 + \frac{1}{k}$. In the case of $k$-dimensional matching (i.e. when the hypergraph is $k$-partite) the integrality gap equals $k - 1$. Chan and Lau \cite{LapChiLau} gave algorithms for these cases. These results and algorithms also extend to the weighted case.
  \item The intersecting family LP for $k$-set packing has integrality gap at most $\frac{k}{3} + 1 + \varepsilon$ (Theorem \ref{thm:IntegralityGap2}). It is not known whether this result also extends to the weighted case.
  \item By the results of \cite{LapChiLau}, there also exists a polynomially sized LP for $k$-set packing with integrality gap at most $\frac{k}{3} + 1 + \varepsilon$. We don't know whether this is also true for the weighted version either.
  \item Also by \cite{LapChiLau}, the Lov\'{a}sz Theta function is at least as strong as the intersecting family and therefore this SDP relaxation has integrality gap at most $\frac{k}{3} + 1 + \varepsilon$.
\end{enumerate}

\subsection{Extending results to the weighted case}\label{subsec:WeightedLP}

\paragraph{Unweighted versus weighted} While the results for the standard LP extend to the weighted case, the results for the intersecting family LP do not extend in an obvious way. This is for the same reason why the unweighted approximation algorithms for $k$-set packing do not easily generalise to the weighted case (c.f. Section \ref{sec:WeightedVSUnweighted}): the local search technique is relying crucially on cardinality. For example, $\mathcal{F}_1(e)$ in the proof of Theorem \ref{thm:IntegralityGap2} need not be an intersecting family in the weighted case: the objective function might increase even by adding less sets than are removed. We elaborate on this in Section \ref{sec:WeightedVSUnweighted}. %It is not obvious how to improve significantly on the integrality gap of the standard LP for weighted $k$-set packing.

\paragraph{Possible research directions} Perhaps there is another way to partition the hyperedges rather than in $\mathcal{F}_1$, $\mathcal{F}_2$ and $\mathcal{F}_{3+}$ that does provide a way to extend the result to the weighted case. For the weighted problem the setting of $\mathcal{F}_1$, $\mathcal{F}_2$ and $\mathcal{F}_{3+}$ does not make sense and nothing can be proved. In the weighted case it seems to make sense to define a $t$-locally optimal solution as a solution where adding at most $t$ new sets and losing any number of sets does not yield an improvement (instead of losing less than $t$ sets in the unweighted case). But perhaps in the weighted case it will prove to be worthwhile to consider the integral optimum solution rather than a $t$-locally optimum.

The proof of the existence of a polynomially sized LP in Theorem \ref{thm:LP2} depends on the existence of small kernels for every intersecting family. If a new result on the integrality gap of some LP for weighted $k$-set packing relies on intersecting families, this result still holds. The SDP relaxation still captures all intersecting family constraints in the weighted case (one can just add weights to the objective function and nonnegativity constraints on them), so if one proves a result for weighted $k$-set packing using these intersecting families, that result immediately extends to the polynomially sized LP and SDP relaxation.

\begin{prob}
Narrow the gap between the integrality gap for relaxations of the weighted and the unweighted $k$-set packing problem.
\end{prob}

\subsection{Smaller bound on integrality gap}\label{subsec:IntegralityGap}

Another way to improve upon the current results is to further decrease the upper bound on the integrality gap on LP relaxations like the intersecting family. This seems likely to be possible, because all unweighted approximation algorithms achieve bounds of $\frac{k+2}{3} + \varepsilon$ or $\frac{k+1}{3} + \varepsilon$ while the current result is $\frac{k}{3} + 1 + \varepsilon$. Perhaps some ideas from these algorithms can be extended to the integrality gap of the LP relaxation.

\begin{prob}
Improve the integrality gap of relaxations for the unweighted $k$-set packing to $\frac{k+2}{3} + \varepsilon$ or better.
\end{prob}

\paragraph{Towards a gap of $\frac{k+1}{3} + \varepsilon$} In particular we would like to point out that it may be worthwhile to see if the idea from the quasi-polynomial time $\frac{k+1}{3} + \varepsilon$ from $\cite{Mastrolilli}$ can be extended. The proof of this approximation guarantee also depends on Lemma \ref{lem:BermanMIS2} as did our new bound of $\frac{k}{3} + 1 + \varepsilon$. Their algorithm uses slightly more crafted ideas, but these do not seem to extend to the integrality gap of the LP in a straightforward way. Roughly speaking they are able to bound $\mathcal{F}_1$ by $\varepsilon |M|$ rather than $|M|$. This is then added twice like in the proof of Theorem \ref{thm:IntegralityGap2}. This effectively decreases the bound by $\frac{2}{3}|M|$ which is exactly the current difference between the results. Perhaps by altering the argument a little bit, one could use the idea of this proof to improve the bound for the integrality gap for the intersecting family LP to, say, $\frac{k+1}{3} + \varepsilon$.

\paragraph{Towards a gap of $\frac{k+2}{3} + \varepsilon$} Also the idea from the $\frac{k+2}{3}$-approximation by Sviridenko and Ward \cite{Sviridenko} might be interesting to take a closer look at. Very roughly speaking they bound $\mathcal{F}_2$ by $2 |M \setminus \mathcal{F}_1|$. This cancels against the bound for $\mathcal{F}_1$ and hence they obtain $3 |\mathcal{F}| \leq k|M| + 2|M|$. A similar idea could perhaps be used to bound the LP-values of the sets $\mathcal{F}_1$ and $\mathcal{F}_2$.

Since the local search techniques have reached their limits in terms of their approximation guarantee, other techniques have to be sought to improve the approximation guarantee. The results on the LP and SDP relaxations for $k$-set packing are scarce and we believe more research in this area could turn out to be fruitful. We make the following conjecture.

\begin{conjecture}
The integrality gap of relaxations for the unweighted $k$-set packing can be bounded by $\frac{k+1}{3} + \varepsilon$.
\end{conjecture}

\section{Improving Berman's weighted approximation}\label{sec:DiscBerman}

\subsection{Generalising $charge$ and claws}\label{subsec:DiscBerman1}

The analysis as presented in Section \ref{sec:ProofBerman} shows that slight changes to the algorithm do not simply constitute an improvement of the approximation guarantee to $\frac{k}{c}$ for some $c > 2$. Let us replace the $\frac{1}{2}$ in the definition of $charge(u,v)$ to $\frac{1}{c}$ and the $\frac{1}{2}$ in the definition of a nice (good) claw to $\frac{1}{c}$.

\paragraph{Approximation guarantee} Lemma \ref{lem:Berman1} can now be easily adapted to hold for any $c>2$. The proof remains exactly the same and \textsc{WishfulThinking} now has an approximation guarantee of $\frac{k}{c}$.

\paragraph{Nice claw improves $w^2(A)$} However, Lemma \ref{lem:Berman2} is no longer true and this can be seen from our simplified proof. In the proof that the talons of a nice claw improve $w^2(A)$ when the definitions of $charge(u,v)$ and of a nice claw have been altered, the analysis remains exactly the same except for the fact that in Equation \eqref{eighth}, the factor 2 in the right-hand side changes into a $c$:

\begin{equation}\label{eighth2}
w^2(u) + w^2(v) \geq cw(u)w(v),
\end{equation}

However, Equation \eqref{eighth2} is only true for all $u$ and $v$ for the values $c=0$ and $c=2$. Since $c=0$ is not possible because $c$ occurs in the denominator in the definition of $charge(u,v)$ and of a nice claw, $w^2(A)$ improves only when $c=2$ when following the current analysis.

It is possible to incorporate an extra constraint in the definition of a nice claw to make sure that Equation \eqref{eighth2} holds for some $c > 2$. Then Lemma \ref{lem:Berman2} is true because of the altered definition. However, the extra constraint now causes trouble in the proof of Lemma \ref{lem:Berman1}, which is then not necessarily true anymore.

\subsection{Generalising the weight function}\label{subsec:DiscBerman2}

Intuitively, what is the crucial point why $w^2$ might behave differently than $w$? In general, looking at the squared weight function (or at $w^c$ for any $c>1$) is slightly more biased towards larger weights. When a vertex has some weight $m$ and we add 1 to its weight, $w$ increases by 1 while $w^2$ increases by $2m+1$. Also the square of the weight of one vertex might be more than the weight of two vertices. $w^2$ prefers one vertex of weight 3 to two vertices of weight 2, while $w$ prefers the two vertices of weight 2 to the single vertex of weight 3. So by guiding the search by $w^2$ rather than $w$, an iteration might decrease the real objective function but it is more difficult to get stuck in an inferior locally optimal solution and hence a better result might be achieved in the end.

One could also try to improve $w^p$ for some $2 < p \in \mathbb{N}$ next to the use of the parameter $c$ rather than 2. Following the same analysis as in Section \ref{sec:ProofBerman}, Equation \eqref{eighth} then translates to
\begin{equation*}
w^p(u) + w^p(v) \geq c w(u) w^{p-1}(v).
\end{equation*}
This is trivially true for $c = 0$ and even $p$, for $p = c = 1$ and for $p = c = 2$. Again the first is not a real option, the second option gives an approximation guarantee of $\frac{k}{1}$ which is not an improvement and the third option is the one that leads to Berman's result. %One could try to choose $c = { p \choose 1 }$, so that the term on the right-hand side cancels when one applies that binomial theorem on $(w(u)+w(v))^p$.
%Using the binomial theorem on $(w(u)+w(v))^p$ cancels the term on the right-hand side by letting $c={p \choose 1}=p$, but yields some other terms that seem to hurt.
%Note that this is also true for $p = c = 1$, which gives the standard $d$-approximation back.
This shows that using the current analysis this is the best possible result, so to improve the approximation guarantee one really needs another analysis or another approach. See also \cite{BermanWeighted2} for other reasoning about generalising Berman's algorithm.

\section{Relation to the independent set problem}\label{sec:IS}

Here is a summary of the relation between $k$-set packing and the independent set problem.

\begin{enumerate}
  \item Finding a maximum set packing is equivalent to finding a maximum independent set in the conflict graph of the set packing instance.
  \item Finding a maximum $k$-set packing can be reduced to finding a maximum independent set in the $k+1$-claw free conflict graph of the $k$-set packing instance. %This last problem is more general than $k$-set packing as there are $k+1$-claw free graphs that do not correspond to any $k$-set packing instance.
  %\item $k$-set packing is a more general problem than the maximum independent set problem on bounded degree graphs.
\end{enumerate}

\subsection{Results on the independent set problem}\label{subsec:ResultsIS}

\paragraph{General graphs} The following results are known for the maximum independent set problem in general. Let $n$ be the number of vertices in the graph.
\begin{enumerate}
  \item The problem is NP-hard and it is hard to approximate within $n^{1 - \varepsilon}$ unless NP-hard problems have randomised polynomial time algorithms \cite{CliqueIsHard}.
  \item There is an approximation algorithm that achieves an approximation guarantee of $\Theta\left( \frac{n}{\log^2 n} \right)$ \cite{ApproxGeneralIS}.
  %\item There exists no constant factor approximation algorithm \cite{noPTAS1} if $\mathcal{P} \neq \mathcal{NP}$.
  %\item An $n^{\frac{1}{4}}$-approximation ratio is out of reach \cite{NoApproxGeneralIS}.
\end{enumerate}
\paragraph{Bounded degree graphs}
When the maximum degree of every vertex is assumed to be bounded by some $\Delta$, approximating the problem becomes considerably easier. The following results are known for the maximum independent set problem on bounded degree graphs.
\begin{enumerate}
  \item Assuming the Unique Games Conjecture \cite{UGC} it is hard to approximate within $O \left( \frac{\Delta}{\log^2 \Delta} \right)$ \cite{Bounded1}.
  %\item For all $\Delta > 3$ these exists an $\varepsilon > 1$ such that it is NP-hard to $(1 + \varepsilon)$-approximate the problem \cite{noPTAS1,NoPTAS3}.
  %\item There exists an $\varepsilon > 0$ such that it is NP-hard to $O \left( \Delta^\varepsilon \right)$-approximate the problem \cite{Communication}.
  \item %It is obvious the simple greedy algorithm finds an independent set of size at least $\frac{n}{\Delta}$.
  The greedy algorithm is an obvious $\Delta$-approximation. Hochbaum \cite{Hochbaum} was the first to give an approximation algorithm with approximation guarantee $\frac{\Delta}{2}$, which was improved to $\frac{\Delta+2}{3}$ \cite{GreedIsGood,HalldorssonLau}. Berman and F\"{u}rer \cite{BermanMIS} give a $\frac{\Delta+3}{5} + \varepsilon$-approximation for even $\Delta$ and a $\frac{\Delta+3.25}{5} + \varepsilon$-approximation for odd $\Delta$, which was slightly improved in \cite{BermanFujito}. Then a major jump came with an $O \left( \frac{\Delta}{ \log \log \Delta} \right)$-approximation, a $\frac{\Delta}{6}(1+o(1))$-approximation, a proof that greedy achieves $\frac{\Delta+2}{3}$ and that Berman and F\"{u}rer really achieved $\frac{\Delta+3}{4}$ \cite{Bounded2}. The currently best result is an $O \left( \frac{\Delta \log \log \Delta}{\log \Delta} \right)$-approximation in polynomial time that also extends to the weighted case \cite{Halperin,SpecialWeight,Vishwanathan}.
\end{enumerate}

\subsection{From bounded degree to claw-free graphs}\label{subsec:DiscIS2}

\paragraph{Comparison} The results on the maximum independent set problem in bounded degree graphs are much stronger than the results on $k$-set packing. For example, the gap between its hardness ($\Omega \left( \frac{\Delta}{\log^2 \Delta} \right)$) and its best approximation guarantee ($O \left( \frac{\Delta \log \log \Delta}{\log \Delta} \right)$) is much smaller ($\Omega(\frac{k}{\log k})$ versus $\frac{k+1}{3} + \varepsilon$). The constraint that the degree is bounded apparently allows much more arguments than the constraint that the size of an independent set in the neighbourhood of every vertex is bounded. %[TODO: What goes wrong if you try to use the SDP approach for d/log d here? What may be a way out?]

\begin{prob}
Narrow the gap between the approximation guarantee of $k$-set packing and the independent set problem on bounded degree graphs.
\end{prob}

\paragraph{Mimicking the bounded degree algorithm} As an example, consider the $O \left( \frac{\Delta \log \log \Delta}{\log \Delta} \right)$-approximation for weighted independent set in bounded degree graphs by Halperin \cite[Section 5]{Halperin}. In a nutshell this solves a semidefinite programming relaxation and partitions the resulting vectors in sets $S_0$, $S_1$ and $S_2$. It then uses the greedy approach on $S_0$ to find an independent set $I_0$, projects and normalises the vectors in $S_1$ and selects some of them to find $I_1$, and just sets $I_2 = S_2$. It then returns the largest weight independent set from $I_0$, $I_1$ and $I_2$. By a good choice of a parameter the approximation guarantee is achieved.

If one tries to use the same approach on the independent set problem in $k+1$-claw free graphs it is the set $S_0$ that is causing trouble. In the bounded degree case it is trivially true that the greedy algorithm produces an independent set $I_0$ of total weight at least $\frac{w(S_0)}{\Delta+1}$ where $w(S_0)$ is the sum of the weights of the vertices in $S_0$. %by just selecting the vertex with the lowest degree and removing it along with all its neighbours, and then iterating on the graph with $\Delta + 1$ vertices less.
Intuitively $S_0$ is already a good structured set so it is not needed to do anything smarter than the greedy algorithm.

However, in the $k+1$-claw free case the greedy algorithm does not have such a sufficient performance guarantee. It has a performance guarantee of $k$, %but a performance guarantee in terms of $w(S_0)$ (or $|S_0|$ in the unweighted case) is needed.
so the value of the greedy solution can be compared to the optimal independent set size in $S_0$. But it can not be compared to the weight (or the cardinality) of $S_0$.

\paragraph{Intuition and research direction} Morally it seems there should not be such a difference between the independent set problem in bounded degree graphs and claw-free graphs. Look at a vertex $v$ in a claw-free graph and at its neighbours $N(v)$. As the maximum size of an independent set in $N(v)$ is at most $k$, one could look at $N(v)$ as the union of $k$ cliques (in relation to the set packing instance: one clique for the sets that all share one of the $k$ elements, modulo some duplicates). And for the independent set problem, a clique is not that different from a vertex: it is only possible to pick one of the vertices. It would be an interesting research direction to see where exactly the analogy with the bounded degree graphs stops, or to somehow change the algorithm for the bounded degree graphs and achieve an improved approximation guarantee for claw-free graphs.

The intuition in this paragraph inclines us to believe the approximation guarantee for $k$-set packing can be brought down further. Perhaps it it not $\frac{k}{\log k}$, which is the best known hardness bound, but we make the following conjecture.

\begin{conjecture}
The approximation guarantee for $k$-set packing can be bounded by some function strictly smaller than $\frac{k}{3}$.
\end{conjecture}

%As such, the $k$-set packing problem generalizes the maximum independent set problem in bounded degree graphs. This is an NP-complete problem, it does not admit a polynomial time approximation scheme \cite{noPTAS1}, and assuming the Unique Games conjecture (see \cite{UGC}) the problem of finding a maximum independent set in a $d$-regular graph cannot be approximated within a factor $O\left(\frac{d}{\log^2{d}}\right)$ \cite{Bounded1}. On the positive side, the greedy approach for this problem was shown to have an approximation ratio of $\frac{d+2}{3}$ \cite{GreedIsGood}. The currently best approximation algorithm achieves an approximation guarantee of $O\left(\frac{d}{\log \log d}\right)$ \cite{Bounded2}. In Chapter \ref{chap:IS} we will elaborate more on the difference between the maximum independent set problem in bounded degree graphs and $k+1$-claw free graphs [TODO].

\section{Unweighted versus weighted $k$-set packing}\label{sec:WeightedVSUnweighted}

\paragraph{General extension} For most problems the weighted version is not much more difficult than the unweighted version. For example, also weighted matchings can be found in graphs in polynomial time, and for a lot of problems adding weights to the LP-formulation does not change anything significantly. However, for $k$-set packing the difference between the weighted version and the unweighted version is nontrivial.

\paragraph{Extending $k$-set packing} As noted in Section \ref{sec:Unweighted}, the analysis of the unweighted case hugely depends on the cardinality of every set. This is because the algorithms rely on local search and analyse a locally optimal solution. However, adding more sets than you remove from your current solution may not be advantageous in the weighted case, as the total weight might decrease while the cardinality of the solution increases. And in the weighted case, it could be the case that the total weight increases when you add less sets than you remove from your current solution. This proves hard to be incorporated in the analysis, and the weights of the sets cannot be handled in a straightforward way. This is why the results for the unweighted case do not easily extend to the weighted case. There have been less results on the weighted case and the algorithms do not immediately follow from the unweighted results, although they all use local search techniques.

\begin{prob}
Improve the approximation guarantee for weighted $k$-set packing.
\end{prob}

\paragraph{Research direction} The best weighted approximation algorithms reduce the problem to the independent set problem on $k+1$-claw free graphs. There are no indications so far that better results can be obtained in the weighted $k$-set packing problem by holding back from this reduction, so this seems a good way to look at the problem. These weighted approximation algorithms achieve approximation guarantees of approximately $\frac{2k}{3}$ \cite{Chandra} and $\frac{k}{2}$ \cite{Berman}. Section \ref{sec:DiscBerman} considered possible ways to improve upon this last algorithm by changing it a little bit, aided by the simplified proof from Chapter \ref{chap:Weighted}. Berman and Krysta \cite{BermanWeighted2} considered what values of $w^\alpha$ are the best for every $k$ when one searches a 2-locally optimal solution and achieve an approximation guarantee of about $\frac{2k}{3}$. A natural extension of this would be to use this alternate weight function in a search for $t$-locally optimal solutions for some $t>2$. Considering such larger improving sets proved to work for the unweighted problem and it trivially works for the weighted case when one searches for improving sets of size $n$; the question is, how large do the improving sets have to be to get an improved approximation guarantee of, say, about $\frac{k}{3}$? Perhaps a value of $t = O(k)$ or $t = O(\log n)$ works. We think it is fruitful to investigate this possibility and either find a better approximation algorithm or an indication that this might not be helpful after all.

We believe a better approximation guarantee could be obtained by increasing the search space and we make the following conjecture.

\begin{conjecture}
The approximation guarantee for weighted $k$-set packing can be bounded by $\frac{k+c}{3} + \varepsilon$ for some fixed $c \geq 0$.
\end{conjecture}

%[TODO done: What are the issues while trying to get k/3+c for the weighted case? May be you can conjecture some algorithm that might work, given the new developments for unweighted case]

%[TODO] discussion of what makes (k+2)/3 bound in Sviridenko et al paper easier than (k+1)/3.
%Is it for the same reason why (k+1)/2 is very easy with Hurken-Schrijver  but k/2 needs much more work. 

% UNCOMMENT HYPERREF PACKAGE WHEN CREATING FINAL PDF

\appendix
\chapter{Parameterized complexity}\label{app:Parameterized}

\begin{table}[htp]
\centering
\begin{tabular}{lllll}
  \toprule
  Complexity           & Reference                               &     &     & Remarks \\
  \midrule
  $O^*((5.7m)^m)$      & Jia, Zhang and Chen \cite{5.7kk}        & R   & 3   & \\
  $O^*(10.88^{3m})$    & Koutis \cite{Koutis1}                   & R   & k   & \\
  $O^*(16^m)$          & Chen et al \cite{SmallColorCoding}      & R   & k   & Also for weighted problem \\
% $O^*(2.52^{3m})$     & Chen et al \cite{SmallColorCoding}      & R   & 3   & \\ OLD: 16^m cites DivideAndConquer and this is added
  $O^*(2^{3m})$        & Koutis \cite{Koutis2}                   & R   & k   & \\
  $O^*(1.493^{3m})$    & Bj\"{o}rklund et al \cite{Bjorklund}    & R   & k   & \\
  \hline
  $O^*(2^{O(m)}(3m)!)$ & Downey and Fellows \cite{DowneyFellows} & D   & 3   & \\
  %$O^*(25.6^{3m})$     & Koutis \cite{Koutis1}                   & D   & k   & Originally $O^*(2^{O(mk)}) \geq O^*(32000^{3m})$ (\cite{SmallColorCoding,GreedyLocalization}) \\
  \multirow{2}{*}{$O^*(25.6^{3m})$} & \multirow{2}{*}{Koutis \cite{Koutis1}} & \multirow{2}{*}{D} & \multirow{2}{*}{k} & Originally $O^*(2^{O(mk)}) \geq O^*(32000^{3m})$ \\
                       &                                         &     &     & (\cite{SmallColorCoding,GreedyLocalization}) \\
  \multirow{2}{*}{$O^*(13.78^{3m})$} & \multirow{2}{*}{Fellows et al \cite{FellowsRosamond}} & \multirow{2}{*}{D} & \multirow{2}{*}{k} & Originally $exp(O(mk)) = O^*(12.7D)^{3m},$ \\
                       &                                         &     &     & $D \geq 10.4$ (\cite{SmallColorCoding,GreedyLocalization}) \\
  $O^*(12.8^{3m})$     & Chen \cite{Chen12.8}                    & D   & k   & \\
  $O^*(12.8^{3m})$     & Liu, Chen and Wang \cite{12.8mk}        & D   & k   & Also for weighted problem \\
  $O^*(12.8^{3m})$     & Chen et al \cite{SmallColorCoding}      & D   & k   & Also for weighted problem \\ % Bjorklund says 4^{3k+o(k)} instead of this [TODO]
  $O^*(7.56^{3m})$     & Wang and Feng \cite{WangFeng2}          & D   & 3   & Also for weighted problem \\
  $O^*(5.44^{3m})$     & Chen and Chen \cite{ChenChen}           & D   & k   & \\
% $O^*(4.68^{3m})$     & Kneis et al \cite{DivideAndColor}       & D   & 3   & \\ Is this right? From Bjorklund. Paper itself does not talk about set packing at all? [TODO]
% $O^*(4^{3m+o(m)})$   & Chen et al \cite{SmallColorCoding}      & D   & 3   & \\ Is this right? From Bjorklund. SmallColorCoding itself says it achieves 12.8^3k, see above [TODO]
  $O^*(4.61^{3m})$     & Liu et al \cite{GreedyLocalization}     & D   & 3   & \\
  $O^*(4^{3m})$        & Chen et al \cite{DivideAndConquer}      & D   & k   & Also for weighted problem \\ % Theorem 5.5, not in SmallColorCoding or mentioned in Bjorklund.
  $O^*(3.523^{3m})$    & Wang and Feng \cite{WangFeng1}          & D   & 3   & \\
  $O^*(32^m)$          & Feng et al \cite{WeightedParameterized} & D   & k   & Also for weighted problem \\
  \bottomrule
\end{tabular}
\caption{Parameterized complexity results of randomised (R) and deterministic (D) algorithms for $3-SP$. Column 4 indicates whether the result only applies to $3-SP$ (3) or whether it is derived from a more general result from $k-SP$ (k).}
\label{tab:3SP}
\end{table}

\begin{table}[htp]
\centering
\begin{tabular}{llll}
  \toprule
  Complexity         & Reference                            &     & Remarks \\
  \midrule
  $O^*(10.88^{mk})$  & Koutis \cite{Koutis1}                & R   & \\ % Bjorklund zegt 5.44 maar ChenChen en Chen Kneis Lu Zse Zhang zeggen 10.88 [TODO]
  $O^*(4^{(k-1)m})$  & Chen et al \cite{DivideAndConquer}   & R   & Also for weighted problem \\
  $O^*(2^{mk})$      & Koutis \cite{Koutis2}                & R   & \\
  $O^*(f(m,k))$      & Bj\"{o}rklund et al \cite{Bjorklund} & R & $f(m,k) \approx \left(\frac{0.11 \cdot 2^m (1 - \frac{1.64}{m})^{1.64-m} m^{0.68} }{ (m-1)^{0.68} }\right)^k$ \\
  \hline
  $O^*(g(k,m))$      & Jia, Zhang and Chen \cite{5.7kk}     & D   & \\ %ChenChen zegt m^k (g(k,m))^{mk}, Chen Kneis Lu Zse Zhang en Feng Liu Lu Wang (r-1)^k ((r-1)k/e)^{k(r-2)} [TODO]
  $O^*(25.6^{mk})$   & Koutis \cite{Koutis1}                & D   & Originally $O^*(2^{O(mk)})$ (\cite{SmallColorCoding,GreedyLocalization}) \\
  $O^*(13.78^{mk})$  & Fellows et al \cite{Fellows}         & D   & Originally $exp(O(mk))$ (\cite{SmallColorCoding,GreedyLocalization}) \\ % Chen Kneis Lu Zse Zhang en ook Feng Liu Lu Wang zegt 2^{5rk-4k} {6(r-1)k+k} choose {rk}, paper zelf zegt n+2^O(k) [TODO]
  $O^*(12.8^{mk})$   & Chen \cite{Chen12.8}                 & D   & \\
  $O^*(12.8^{mk})$   & Liu, Chen and Wang \cite{12.8mk}     & D   & Also for weighted problem \\
  $O^*(5.44^{mk})$   & Chen and Chen \cite{ChenChen}        & D   & \\
  $O^*(4^{mk})$      & Chen et al \cite{DivideAndConquer}   & D   & Also for weighted problem \\
  $O^*(2^{(2k-1)m})$ & Feng et al \cite{WeightedParameterized} & D & Also for weighted problem \\
  \bottomrule
\end{tabular}
\caption{Parameterized complexity results of randomised (R) and deterministic (D) algorithms for $k-SP$.}
\label{tab:kSP}
\end{table} 

\backmatter

\bibliographystyle{alpha}
\newpage
\phantomsection
\addcontentsline{toc}{chapter}{\numberline{}Bibliography}
\thispagestyle{plain}
\bibliography{References}

\end{document}